\documentclass{amsart}

\usepackage{amsmath}
\usepackage{amssymb}
\usepackage{amsthm}
\usepackage{amscd}
\usepackage[all]{xy}
\usepackage{url}

\newtheorem{thm}{Theorem}[section]
\newtheorem{prop}[thm]{Proposition}
\newtheorem{lem}[thm]{Lemma}
\newtheorem{cor}[thm]{Corollary}
\newtheorem{assumption}[thm]{Assumption}

\theoremstyle{definition}
\newtheorem{dfn}[thm]{Definition}

\theoremstyle{remark}
\newtheorem{rem}[thm]{Remark}
\newtheorem*{acknowledgement}{Acknowledgements}

\newcommand{\R}{\mathbb{R}}
\newcommand{\Z}{\mathbb{Z}}

\newcommand{\A}{\mathcal{A}}
\newcommand{\B}{\mathcal{B}}
\newcommand{\F}{\mathcal{F}}

\newcommand{\Fr}{\mathrm{Fr}}
\newcommand{\Geod}{\mathrm{Geod}}

\title{Magnitude homology of geodesic space}

\author{Kiyonori Gomi}

\address{
Department of Mathematics, 
Institute of Science Tokyo, 
2-12-1 Ookayama, Meguro-ku, Tokyo, 152-8551, Japan.}

\email{kgomi@math.titech.ac.jp}

\subjclass[2010]{55N35, 51F99, 18G40}

\keywords{Magnitude homology, metric space, geodesic metric space.}

\begin{document}

\begin{abstract}
This paper studies the magnitude homology groups of geodesic metric spaces. We start with a description of the second magnitude homology of a general metric space in terms of the zeroth homology groups of certain simplicial complexes. Then, on a geodesic metric space, we interpret the description by means of geodesics. The third magnitude homology of a geodesic metric space also admits a description in terms of a simplicial complex. Under an assumption on a metric space, the simplicial description allows us to introduce an invariant of third magnitude homology classes as an intersection number. Finally, we provide a complete description of all the magnitude homology groups of a geodesic metric space which fulfils a certain non-branching assumption.
\end{abstract}

\maketitle

\tableofcontents


\section{Introduction}
\label{sec:introduction}

The \textit{magnitude} \cite{L} of a metric space $(X, d)$ is an invariant which ``effectively'' counts the number of points on $X$. As its categorification, the \textit{magnitude homology} is introduced \cite{H-W,L-S}. Since its introduction, a number of works have appeared recently. For example, a problem raised in \cite{H-W,L-S} is the existence of torsions in magnitude homology groups. This is solved affirmatively in \cite{K-Y}. Another problem is to introduce a blurred version of magnitude homology \cite{L-S}, which is solved in \cite{Ott}.

Because of its origin, the magnitude homology is defined in a rather categorical or algebraic way, and its geometric meaning is much less clear. This contrasts with the singular homology of a topological space, which appeals more or less to our geometric intuition. To uncover geometric meanings of the magnitude homology is the motivation of the present paper. 

\medskip

Regarding low degree magnitude homology groups, some results are known \cite{H-W,L-S}. To review these results, we denote by $H^\ell_n(X)$ the $n$th magnitude homology group of a metric space $(X, d)$ with its ``characteristic length'' $\ell$. In general, a magnitude homology class in $H^\ell_n(X)$ is represented by a linear combination of ``chains'' $\langle x_0, \cdots, x_n \rangle$ of points $x_0 \neq x_1 \neq \cdots \neq x_n$ on $X$ constrained as
$$
d(x_0, x_1) + d(x_1, x_2) + \cdots + d(x_{n-1}, x_n) = \ell
$$
by the (non-negative) real number $\ell$. Then it is known \cite{H-W,L-S} that
\begin{align*}
H^\ell_0(X) &=
\left\{
\begin{array}{ll}
\bigoplus_{x \in X} \Z \langle x \rangle, & (\ell = 0) \\
0, & (\ell > 0)
\end{array}
\right.
&
H^\ell_1(X) &=
\left\{
\begin{array}{ll}
0, & (\ell = 0) \\
\bigoplus_{\stackrel{d(a, b) = \ell}{a < \nexists x < b}}
\Z \langle a, b \rangle, & (\ell > 0)
\end{array}
\right.
\end{align*}
where $a < \nexists x < b$ means that there is no point $x$ such that $a \neq x \neq b$ and $d(a, x) + d(x, b) = d(a, b)$. Therefore the zeroth magnitude homology is generated by the points on $X$. Notice that the existence of a point $x$ such that $a \neq x \neq b$ and $d(a, x) + d(x, b) = d(a, b)$ for any $a, b \in X$ is the characterization that $X$ is \textit{Menger convex}. \textit{Geodesic metric spaces} are examples of Menger convex metric spaces. Then the description above tells us that a Menger convex space $X$ has the trivial first homology $H^\ell_1(X) = 0$ for any $\ell$. For the second magnitude homology, a description in terms of chains is also given in \cite{L-S} when $X$ is a \textit{geodetic} metric space satisfying some additional assumptions.

\medskip

Since the zeroth and first magnitude homology groups of a metric space are understood as above, the present paper begins with a study of the second and third magnitude homology groups, in the case that the metric space is in particular geodesic. Before the statements of our results, we notice two basic facts. The first fact is that the magnitude homology $H^\ell_n(X)$ of a metric space $(X, d)$ admits a direct sum decomposition
$$
H^\ell_n(X) = \bigoplus_{a, b \in X} H^\ell_n(a, b),
$$
where $H^\ell_n(a, b)$ is made of chains $\langle x_0, \cdots, x_n \rangle$ whose ``endpoints'' are $x_0 = a$ and $x_n = b$. Thus, we will mainly concerned with the direct summand $H^\ell_n(a, b)$ in this paper. The second fact is that the direct summand $H^\ell_n(a, b)$ is trivial whenever $\ell < d(a, b)$, and $H^\ell_n(a, b)$ with $\ell = d(a, b)$ and $n \ge 2$ admits a description
$$
H^\ell_n(a, b) \cong \overline{H}_{n-2}(A(a, b))
$$
in terms of the reduced simplicial homology of a simplicial complex $A(a, b)$. The vanishing in the case of $\ell < d(a, b)$ follows directly from an estimate of the length of chains, whereas the description in the case of $\ell = d(a, b)$ is a result in \cite{K-Y} (see \S\S\ref{subsec:simplicial_complex_in_magnitude_homology} for detail). 

Now, our result about the second magnitude homology (Theorem \ref{thm:main_description}) is:

\begin{thm}
Let $(X, d)$ be a metric space, and $\ell$ a real number. For any $a, b \in X$, we have
\begin{align*}
H^\ell_2(a, b)
&\cong
\left\{
\begin{array}{ll}
0, & (\ell < d(a, b)) \\
\overline{H}_0(A(a, b)), & (\ell = d(a, b)) \\
H_0(B^\ell(a, b)), & (\ell > d(a, b))
\end{array}
\right.
\end{align*}
where $H_0(B^\ell(a, b))$ is the zeroth homology group of a simplicial complex $B^\ell(a, b)$. 
\end{thm}

We remark that there are other ways to represent the magnitude homology as developed in \cite{A-Iv,A-Iz,H-W,K-Y,T-Y} for example.

In general, the zeroth (reduced and unreduced) homology group of a simplicial complex is torsion free. Thus, we can conclude that the second magnitude homology $H^\ell_2(X)$ is torsion free for any metric space $X$ and any real number $\ell$. As is mentioned, the existence of torsions in magnitude homology is shown in \cite{K-Y}. More precisely, the argument in \cite{K-Y} tells that the magnitude homology group $H^\ell_n(X)$ can contain a torsion subgroup if $n \ge 3$. In view of this fact, a problem raised by Hepworth \footnote{\url|https://golem.ph.utexas.edu/category/2018/04/torsion_graph_magnitude_homolo.html|} is whether the second magnitude homology admits torsions or not. Our result solves this problem.

The part $H^\ell_2(a, b) \cong H_0(B^\ell(a, b))$ appearing in the case of $\ell > d(a, b)$ is trivial if $(X, d)$ is a geodesic space. Under this assumption, $\overline{H}_0(A(a, b))$ admits an interpretation by means of certain equivalence classes of geodesics on $X$. Concretely, let $\Geod(a, b)$ denote the set of geodesics joining $a$ to $b$. We can then generate an equivalence relation on $\Geod(a, b)$ by a relation $f \sim g$ for $f, g \in \Geod(a, b)$ which exists only when $f$ and $g$ intersect at a point other than $a$ or $b$. We write $\Z[\pi_0(\Geod(a, b))]$ for the free abelian group generated by the equivalence classes of geodesics in $\Geod(a, b)$. This group has a surjective homomorphism $\epsilon : \Z[\pi_0(\Geod(a, b))] \to \Z$ induced from $\epsilon(f) = 1$ for $f \in \Geod(a, b)$. Using these notations, we can describe the direct summand $H^\ell_2(a, b)$ of the magnitude homology of a geodesic space $(X, d)$ as follows (Corollary \ref{cor:second_homology_on_geodesic_space}):
$$
H^\ell_2(a, b) \cong
\left\{
\begin{array}{ll}
\mathrm{Ker}[\epsilon : \Z[\pi_0(\Geod(a, b))] \to \Z], & (d(a, b) = \ell) \\
0. & (d(a, b) \neq \ell)
\end{array}
\right.
$$

\medskip

For the third magnitude homology, we have the following result (Theorem \ref{thm:third_homology_on_geodesic_space}).

\begin{thm} 
Let $(X, d)$ be a geodesic space, and $\ell$ a real number. For any $a, b \in X$, there is an isomorphism 
$$
H^\ell_3(a, b) \cong
\left\{
\begin{array}{ll}
H_1(A(a, b)), & (\ell = d(a, b)) \\
0. & (\ell \neq d(a, b))
\end{array}
\right.
$$
\end{thm}

Using the above description, we can prove the vanishing of the third magnitude homology, if $X$ is a \textit{uniquely geodesic space} (Corollary \ref{cor:third_homology_on_geodesic_space_simple_version}), or more generally, if $X$ is a geodesic space satisfying a certain non-branching assumption (Corollary \ref{cor:third_homology_on_geodesic_space_general_version}). This assumption (Assumption \ref{assumption:non_branching}) is that: for any $a, b \in X$, if two geodesics $f$ and $g$ joining $a$ to $b$ share a point other than $a$ or $b$, then they coincide: $f = g$. This assumption is equivalent to that each equivalence class in $\pi_0(\Geod(a, b))$ is represented by a unique geodesic. It should be noticed that a connected and complete Riemannian manifold satisfies the assumption. Thus, for example, we get the vanishing of the third magnitude homology of the circle with geodesic metric, as shown in \cite{G}.

Using the description also, we introduce an invariant of third magnitude homology classes in $H^\ell_3(a, b)$ with $\ell = d(a, b)$. The key to the formulation of this invariant is that an element in $H^\ell_3(a, b) \cong H_1(A(a, b))$ with $\ell = d(a, b)$, being a first homology class of a simplicial complex, can be represented by a closed edge path of the simplicial complex $A(a, b)$. Under an assumption (Assumption \ref{assumption:for_intersection_number}), which seems somewhat strong, we can associate to the edge path a geometric path in $X$. Then, counting the intersection number of this geometric path and a geodesic $f$ joining $a$ to $b$, we define an invariant of the third magnitude homology classes, which gives rise to a homomorphism (Corollary \ref{cor:intersection_number})
$$
\nu_f : H^\ell_3(a, b) \to \Z.
$$
By an example, this homomorphism turns out to be non-trivial in general. For the time being, the dependence of $\nu_f$ on $f$ and relationship between $\nu_f$ and the \textit{magnitude cohomology} \cite{H} are unknown. These issues deserve future works.

\medskip

In the results about the second and third magnitude homology groups so far, the descriptions become simple when a geodesic space satisfies the non-branching assumption (Assumption \ref{assumption:non_branching}). Remarkably, this assumption leads to a complete description of all the magnitude homology groups (Theorem \ref{thm:complete_description_under_non_branching}).

\begin{thm} \label{thm:main}
Let $(X, d)$ be a geodesic space satisfying Assumption \ref{assumption:non_branching}, namely, for any $a, b \in X$, if two geodesics $f$ joining $a$ to $b$ share a point other than $a$ or $b$, then $f = g$. Then the $n$th magnitude homology group $H^\ell_n(X)$ of length $\ell > 0$ admits the following description.
\begin{enumerate}
\item[(a)]
If $n$ is odd, then $H^\ell_n(X) = 0$ for any $\ell$.

\item[(b)]
If $n = 2q$ is even, then
$$
H^\ell_n(X)
\cong
\bigoplus_{\ell_1, \ldots, \ell_q}
\bigoplus_{\varphi_0, \ldots, \varphi_q}
\bigoplus_{f_1, \ldots, f_q}
\Z (f_1, \ldots, f_q).
$$
In the above, the direct sum is taken over:
\begin{itemize}
\item
positive real numbers $\ell_1, \ldots, \ell_q$ such that $\ell = \ell_1 + \cdots + \ell_q$, 

\item
points $\varphi_0, \ldots, \varphi_q \in X$ such that $d(\varphi_{i-1}, \varphi_i) = \ell_i$ for $ = 1, \ldots, q$,

\item
geodesics $f_i \in \Geod(\varphi_{i-1}, \varphi_i)$ such that $f_i \neq \overline{f}_i$ for $i = 1, \ldots, q$, 
\end{itemize}
where $\overline{f}_i \in \Geod(\varphi_{i-1}, \varphi_i)$ are geodesics chosen arbitrarily as references. We mean by $\Z(f_1, \ldots, f_q)$ the free abelian group of rank $1$ generated by the formal symbol $(f_1, \ldots, f_q)$.
\end{enumerate}
\end{thm}

We can derive a number of results from the theorem. An immediate consequence is that the magnitude homology of a geodesic space is torsion free, provided the non-branching assumption (Corollary \ref{cor:complete_description_under_non_branching:torsion}). Another immediate consequence is the vanishing of positive degree magnitude homology groups in the case where the metric space $X$ in question is a uniquely geodesic space (Corollary \ref{cor:vanishing_on_uniquely_geodesic_space}). This covers the vanishing in the case where $X$ is:

\begin{itemize}
\item
a convex subset of $\R^N$ with Euclidean distance \cite{K-Y}, 

\item
a connected and complete Riemannian manifold which is uniquely geodesic as a metric space \cite{Jub}, and 

\item
a geodesic $\mathrm{CAT}(\kappa)$-space with $\kappa \le 0$ \cite{A}.

\end{itemize}
If possible geodesics are known, then we can explicitly describe the magnitude homology. For example, we can determine the magnitude homology of the circle $S^1$ of radius $r > 0$ with geodesic metric as follows: 
$$
H^\ell_n(S^1)
\cong
\left\{
\begin{array}{ll}
\Z[S^1], & (\mbox{$n \ge 0$ even and $\ell = \pi r n/2$}) \\
0, & (\mbox{otherwise})
\end{array}
\right.
$$
where $\Z[S^1]$ is the abelian group generated by points on $S^1$. Note that, so far, $H^\ell_n(S^1)$ is determined only when $n$ or $\ell$ is small \cite{G,K-Y,L-S}. Based on the universal coefficient theorem \cite{H} and the explicit basis, we can also determine the magnitude cohomology ring of $S^1$, the detail of which is left to interested readers.

\medskip

Thanks to Theorem \ref{thm:main} (Theorem \ref{thm:complete_description_under_non_branching}), the magnitude homology of a geodesic space is well understood under the non-branching assumption (Assumption \ref{assumption:non_branching}). However, as mentioned above, the (third) magnitude homology of a geodesic space is generally non-trivial. It is not clear to what extent the description under the assumption can be generalized.

\medskip

To conclude the introduction, we point out that our results have the tendency that some assumptions about the uniqueness of geodesics simplify the magnitude homology groups. The extreme case is the vanishing of positive degree magnitude homology of a uniquely geodesic space. Therefore a possible geometric intuition about the magnitude homology of a geodesic space would be: ``The more geodesics are unique, the more magnitude homology is trivial''.

\bigskip

The present paper is organized as follows: In \S\ref{sec:preliminary}, we review the definition of the magnitude homology, and the fact that $H^\ell_n(a, b)$ with $\ell = d(a, b)$ is identified with the simplicial homology $\overline{H}_{n-2}(A(a, b))$. We also review the \textit{smoothness spectral sequence} \cite{G}, which is applied to the proof of our main descriptions (Theorem \ref{thm:main_description}, Theorem \ref{thm:third_homology_on_geodesic_space} and Theorem \ref{thm:complete_description_under_non_branching}). After that elementary facts about geodesic spaces are reviewed. Then, in \S\ref{sec:second_magnitude_homology}, we study the second magnitude homology of general metric spaces and geodesic metric spaces. \S\ref{sec:third_magnitude_homology} is devoted to the third magnitude homology of a geodesic metric space and its invariant. Finally, in \S\ref{sec:higher_magnitude_homology}, we provide the complete description of the magnitude homology under Assumption \ref{assumption:non_branching}.

\bigskip

\begin{acknowledgement}
I would like to thank Yasuhide Numata, Masahiko Yoshinaga, Shin-ichi Ohta, Yasuhiko Asao and Beno\^{i}t Jubin for helpful discussions and correspondence. I would also like to thank the anonymous reviewer for the comments that lead to improvements. This work is supported by JSPS Grant-in-Aid for Scientific Research on Innovative Areas "Discrete Geometric Analysis for Materials Design": Grant Number JP17H06462.
\end{acknowledgement}


\section{Preliminary}
\label{sec:preliminary}


\subsection{Magnitude homology}

We review here the definition of the magnitude homology. The basic references are \cite{H-W,L-S}, although we use some different notations.

\medskip

Let $X$ be a set. For a non-negative integer $n$, an \textit{$n$-chain} $\langle x_0, \cdots, x_n \rangle$ is defined to be a sequence of points $x_0, \cdots, x_n$ on $X$. An $n$-chain $\langle x_0, \cdots, x_n \rangle$ is said to be \textit{proper} if the adjacent points in the sequence are pairwise distinct: $x_0 \neq x_1 \neq \cdots \neq x_n$. We write $\widehat{P}_n(X)$ for the set of $n$-chains and $P_n(X) \subset \widehat{P}_n(X)$ that of proper $n$-chains. 

Suppose that $X$ is equipped with a distance $d : X \times X \to \R$, so that $(X, d)$ is a metric space. We then define the \textit{length} of an $n$-chain $\langle x_0, \cdots, x_n \rangle \in \widehat{P}_n(X)$ by
$$
\ell(\langle x_0, \cdots, x_n \rangle) 
= d(x_0, x_1) + \cdots + d(x_{n-1}, x_n).
$$
We denote by $\widehat{P}^\ell_n(X) \subset \widehat{P}_n(X)$ the set of $n$-chains of length $\ell$, and by $P^\ell_n(X) \subset P_n(X)$ that of proper $n$-chains of length $\ell$. 

For two (distinct) points $x, z \in X$ given, we say that \textit{$y \in X$ is (strictly) between $x$ and $z$}, and write $x < y < z$, if $x \neq y \neq z$ and $d(x, y) + d(y, z) = d(x, z)$. The notation $x \not< y \not< z$ means that $y$ is not between $x$ and $z$.

For a proper $2$-chain $\langle x_0, x_1, x_2 \rangle$, we have $\ell(\langle x_0, x_1, x_2 \rangle) = d(x_0, x_1)$ if and only if $x_0 < x_1 < x_2$. However, for a proper $3$-chain $\langle x_0, x_1, x_2, x_3 \rangle$, the two relations $x_0 < x_1 < x_2$ and $x_1 < x_2 < x_3$ (which we often write $x_0 < x_1 < x_2 < x_3$) do not generally imply $\ell(\langle x_0, x_1, x_2, x_3 \rangle) = d(x_0, x_3)$. A $3$-chain $\langle x_0, x_1, x_2, x_3 \rangle$ such that $x_0 < x_1 < x_2 < x_3$ and $\ell(\langle x_0, x_1, x_2, x_3 \rangle) > d(x_0, x_3)$ is called a \textit{$4$-cut} \cite{K-Y,L-S}. It is  straight to see that a proper $3$-chain $\langle x_0, x_1, x_2, x_3 \rangle \in P_3(X)$ such that $x_0 < x_1 < x_2 < x_3$ is a $4$-cut if and only if $x_0 \not< x_2 \not< x_3$ (or equivalently $x_0 \not< x_1 \not< x_3$). 

\medskip

For a metric space $(X, d)$ and a non-negative integer $n$, we let $C^\ell_n(X)$ be the free abelian group generated by proper $n$-chains of length $\ell$. For a negative integer $n$, we set $C^\ell_n(X) = 0$. For $n \ge 1$, we define a homomorphism $\partial : C^\ell_n(X) \to C^\ell_{n-1}(X)$ to be the linear extension of 
$$
\partial \langle x_0, \cdots, x_n \rangle
= \sum_{\stackrel{1 \le i \le n-1}{x_{i-1} < x_i < x_{i+1}}}
(-1)^i \langle x_0, \cdots, x_{i-1}, x_{i+1}, \cdots, x_n \rangle.
$$
If we introduce the notation
$$
\partial_i \langle x_0, \cdots, x_n \rangle
= 
\left\{
\begin{array}{ll}
\langle x_0, \cdots, x_{i-1}, x_{i+1}, \cdots, x_n \rangle, &
(x_{i-1} < x_i < x_{i+1}) \\
0, & (x_{i-1} \not< x_i \not< x_{i+1})
\end{array}
\right.
$$
then we can write $\partial = \sum_{i = 1}^{n-1}(-1)^i\partial_i$. For $n \le 0$, we define $\partial : C^\ell_n(X) \to C^\ell_{n-1}(X)$ to be trivial. It turns out that $(C^\ell_*(X), \partial)$ is a chain complex, called the magnitude chain complex. Its homology $H^\ell_*(X)$ is the magnitude homology of $(X, d)$.

\bigskip

The magnitude complex and the magnitude homology admit intrinsic direct sum decompositions: Given two points $a, b \in X$, we define $P^\ell_n(a, b) \subset P^\ell_n(X)$ to be the set of proper $n$-chains $\langle x_0, \cdots, x_n \rangle$ of length $\ell$ such that $x_0 = a$ and $x_n = b$. We let $C^\ell_n(a, b)$ be the free abelian group generated by chains in $P^\ell_n(a, b)$. We clearly have the direct sum decomposition
$$
C^\ell_n(X) = \bigoplus_{a, b \in X} C^\ell_n(a, b).
$$
It turns out that the boundary map $\partial$ on the magnitude complex $C^\ell_*(X)$ restricts to one on $C^\ell_*(a, b)$. In this way, we get a subcomplex $(C^\ell_*(a, b), \partial)$ for each $a, b \in X$, and the direct sum decomposition of the magnitude complex. As a result, we also have the direct sum decomposition of the magnitude homology
$$
H^\ell_n(X) = \bigoplus_{a, b \in X} H^\ell_n(a, b).
$$

\begin{rem} \label{rem:notation_in_noncommutative_polynomial}
One can regard $C^\ell_n(X)$ as a subgroup of the non-commutative polynomial ring generated by points in $X$. In this ring, an element such as $x(y_1 + y_2)z = xy_1z + xy_2z$ makes sense. In accordance with this fact, we will use later notations such as
$$
\langle x, y_1 + y_2, z \rangle = 
\langle x, y_1, z \rangle + \langle x, y_2, z \rangle.
$$

\end{rem}


\subsection{A simplicial complex in magnitude homology}
\label{subsec:simplicial_complex_in_magnitude_homology}

Let $(X, d)$ be a metric space, and $\ell$ a real number. For two points $a, b \in X$, let $H^\ell_n(a, b)$ be the direct summand of the magnitude homology. The length of any chain of the form $\langle a, x_1, \cdots, x_{n-1}, b \rangle$ has the lower bound $\ell(\langle a, x_1, \cdots, x_{n-1}, b \rangle) \ge d(a, b)$. From this, it clearly follows that $H^\ell_n(a, b) = 0$ if $\ell < d(a, b)$. In the case that $\ell = d(a, b)$, the direct summand $H^\ell_n(a, b)$ is described as the homology of a simplicial complex \cite{K-Y}. We review here this simplicial complex. (As a basic reference of simplicial complexes, we refer to \cite{Spa}.)

The key to the description is a property of \textit{geodesically simple chains} \cite{K-Y}: Let $\langle a, x_1, \cdots, x_{n-1}, b \rangle \in C^\ell_n(a, b)$ be a proper $n$-chain of length $\ell$. Under the assumption $\ell = d(a, b)$, we have $a < x_1 < \cdots < x_{n-1} < b$. Furthermore, we have $a < x_{i_1} < \cdots < x_{i_k} < b$ and hence $\langle a, x_{i_1}, \cdots, x_{i_{k-1}}, b \rangle \in C^\ell_k(a, b)$ for any $i_1, \ldots, i_{k-1}$ such that $1 \le i_1 < \cdots < i_{k-1} \le n-1$.

\begin{dfn}
Let $(X, d)$ be a metric space, and $a, b \in X$ distinct two points. We define a simplicial complex $A(a, b)$ as follows:
\begin{itemize}
\item[(a)]
A vertex ($0$-simplex) of $A(a, b)$ is a point $x \in X$ such that $a < x < b$, or equivalently $a \neq x \neq b$ and $d(a, x) + d(x, b) = d(a, b)$.

\item[(b)]
A $(p-1)$-simplex $\{ x_1, \cdots, x_p \}$ of $A(a, b)$ consists of vertices $x_1, \cdots, x_p$ such that $a \neq x_1 \neq \cdots \neq x_p \neq b$ and
$$
d(a, x_1) + d(x_1, x_2) + \cdots + d(x_{p-1}, x_p) + d(x_p, b) = d(a, b).
$$
(Hence $a < x_1 < \cdots < x_p < b$ holds true.)

\item[(c)]
For a $(p-1)$-simplex $\{ x_1, \cdots, x_p \}$ of $A(a, b)$, we give a total order $x_1 \prec \cdots \prec x_p$ of vertices $x_1, \cdots, x_p$ when they satisfy $a < x_1 < \cdots < x_p < b$. By the equivalence class of the total order, the $(p-1)$-simplex is oriented. 
\end{itemize}
\end{dfn}

Let $(C_*(A(a, b)), \partial)$ be the oriented chain complex of the simplicial complex $A(a, b)$. This chain complex has the standard augmentation $\epsilon : C_0(A(a, b)) \to \Z$. We write $\overline{H}_n(A(a, b))$ for the $n$th reduced homology of $A(a, b)$, which agrees with the unreduced homology $H_n(A(a, b))$ if $n \neq 0$. In the case that $a = b$, we put $\overline{H}_n(A(a, b)) = 0$ for all $n$ as a convention. The following description is essentially due to \cite{K-Y} (Theorem 4.4).

\begin{prop} \label{prop:length_is_the_lower_bound}
Let $(X, d)$ be a metric space, $\ell$ a real number, and $a, b \in X$ two points such that $d(a, b) = \ell$. For any $n \ge 2$, we have an isomorphism of groups
$$
H^\ell_n(a, b) \cong \overline{H}_{n-2}(A(a, b)).
$$
\end{prop}

\begin{proof}
If $d(a, b) = \ell = 0$, then we readily see $H^\ell_n(a, a) = 0$ for $n > 0$. Hence the proposition holds true by our convention $\overline{H}_{n-2}(A(a, b)) = 0$. Under the assumption $d(a, b) = \ell > 0$, we have $C^\ell_n(a, b) = 0$ for $n \le 0$, and
\begin{align*}
C^\ell_1(a, b) &= \Z \langle a, b \rangle, &
C^\ell_n(a, b)
&= \bigoplus_{\stackrel{a \neq x_1 \neq \cdots \neq x_{n-1} \neq b}
{\ell(\langle a, x_1, \cdots, x_{n-1}, b \rangle) = d(a, b)}}
\Z \langle a, x_1, \cdots, x_{n-1}, b \rangle
\end{align*}
for $n \ge 1$. By definition, the group $C_{p-1}(A(a, b))$ in the oriented chain complex of $A(a, b)$ is generated by the oriented $(p-1)$-simplices $[x_1 \cdots x_p]$ in $A(a, b)$. The augmentation is $\epsilon([x]) = 1$ and the boundary of $[x_1 \cdots x_p]$ is
$$
\partial [ x_1 \cdots x_p] 
= \sum_{i = 1}^p(-1)^i[x_1 \cdots x_{i-1}x_{i+1} \cdots x_p].
$$ 
There are homomorphisms $\phi_p$ which make the following diagram commutative:
$$
\begin{CD}
C^\ell_1(a, b) @<{\partial}<< C^\ell_2(a, b) @<{\partial}<< C^\ell_3(a, b)
@<{\partial}<< \cdots \\
@V{\phi_1}VV @V{\phi_2}VV @V{\phi_3}VV @. \\
\Z @<{\epsilon}<< C_0(A(a, b)) @<{\partial}<< C_1(A(a, b)) 
@<{\partial}<< \cdots.
\end{CD}
$$
Actually, if we put $\phi_1(\langle a, b \rangle) = 1$ and $\phi_{p+1}(\langle a, x_1, \cdots, x_p, b \rangle) = [x_1 \cdots x_p]$, then $\partial \phi_{p+1} = \phi_p \partial$. Since each $\phi_p$ is bijective, they together form an isomorphism between the augmented chain complex $C_*(A(a, b)) \to \Z$ and the chain complex $C^\ell_{*+2}(a, b)$ augmented by $\partial : C^\ell_2(a, b) \to C^\ell_1(a, b)$. Consequently, $\phi_p$ induces the isomorphism of homology groups in the proposition. 
\end{proof}

As is pointed out \cite{K-Y}, we have the following criterion for the vanishing of the reduced simplicial homology of $A(a, b)$.

\begin{lem} \label{lem:vanishing_of_homology_in_totally_ordered_case}
If the vertices in $A(a, b)$ are totally ordered, then $\overline{H}_n(A(a, b)) = 0$ for all $n \in \Z$.
\end{lem}

\begin{proof}
The lemma follows from the fact that the simplicial complex $A(a, b)$, being the order complex of a totally ordered set, is a simplex.
\end{proof}


\subsection{Smoothness spectral sequence}
\label{subsec:spectral_sequence}

We here review the smoothness filtration of the magnitude complex and its associated spectral sequence \cite{G}.

\medskip

Let $(X, d)$ be a metric space, and $\ell$ a real number. For a proper $n$-chain $\gamma = \langle x_0, \cdots, x_n \rangle \in P^\ell_n(X)$ of length $\ell$, a point $x_i$, ($i = 0, \cdots, n$) in the sequence $x_0, \cdots, x_n$ is said to be \textit{smooth} if $x_{i-1} < x_i < x_{i+1}$, and \textit{singular} otherwise \cite{K-Y}. (In \cite{Jub}, \textit{straight} is used for smooth, and \textit{crooked} for singular.) We write $\sigma(\gamma)$ for the number of smooth points in a proper chain $\gamma$. We can then define $F_pP^\ell_n(X) \subset P^\ell_n(X)$ to be the subset consisting of proper $n$-chains $\gamma$ such that $\sigma(\gamma) \le p$, and $F_pC^\ell_n(X)$ to be the free abelian group generated by chains in $F_pP^\ell_n(X)$. We can see that the subgroup $F_pC^\ell_n(X) \subset C^\ell_n(X)$ forms a subcomplex $(F_pC^\ell_*(X), \partial)$ of $(C^\ell_*(X), \partial)$. We thus get an increasing filtration, called the smoothness filtration of the magnitude chain complex
$$
\cdots \subset 
F_{p-1}C^\ell_* \subset
F_pC^\ell_* \subset
F_{p+1}C^\ell_* \subset 
\cdots.
$$
The associated spectral sequence, which we call the smoothness spectral sequence, converges to a graded quotient of the magnitude homology
$$
E^r_{p, q} \Longrightarrow H_*^\ell(X).
$$

The $E_1$-term of the spectral sequence $E^1_{p, q} = F_pC^\ell_{p+q}(X)/F_{p-1}C^\ell_{p+q}(X)$ is identified with the free abelian group generated by proper $(p+q)$-chains of length $\ell$ which contain exactly $p$ smooth points. If $\gamma = \langle x_0, \cdots, x_{p+q} \rangle$ is such a chain, then the first differential $d^1 : E^1_{p, q} \to E^1_{p-1, q}$ acts as
$$
d^1\gamma = \sum_{i} (-1)^i
\langle x_0, \cdots, x_{i-1}, x_{i+1}, \cdots, x_{p+q} \rangle,
$$
where $i$ runs over $1, \cdots, p+q-1$ such that $x_{i-1} < x_i < x_{i+1}$ and 
$$
\sigma(\langle x_0, \cdots, x_{i-1}, x_{i+1}, \cdots, x_{p+q} \rangle) 
= p - 1. 
$$
Because any proper $n$-chain $\gamma = \langle x_0, \cdots, x_n \rangle \in P^\ell_n(X)$ with $n \ge 1$ contains two singular points $x_0$ and $x_n$, it is clear that $E^1_{p, 0} = 0$ for $p \ge 1$.

Using a notion of \textit{frame} \cite{K-Y}, we can introduce a direct sum decomposition of the $E^1$-term: Given a proper $(p+q)$-chain $\gamma = \langle x_0, \cdots, x_{p+q} \rangle \in P^\ell_{p+q}(X)$ with $\sigma(\gamma) = p$, the frame of $\gamma$ is defined as the $q$-chain
$$
\Fr(\gamma) = \langle x_0, x_{i_1}, \cdots, x_{i_{q-1}}, x_n \rangle 
\in \widehat{P}_q(X)
$$
given by removing the $p$ smooth points from $\gamma$. Notice that $\Fr(\gamma)$ may be an improper chain and its length may be less than $\ell$. For a $q$-chain $\varphi \in \widehat{P}_q(X)$, we write $P^\ell_{p+q}(\varphi) \subset P^\ell_{p+q}(X)$ for the subset of proper $(p+q)$-chains of length $\ell$ whose frames are $\varphi$. We also write $C^\ell_{p+q}(\varphi)$ for the free abelian group generated by $\gamma \in P^\ell_{p+q}(\varphi)$. By design, $(C^\ell_{*+q}(\varphi), d^1)$ is a subcomplex of $(E^1_{*, q}, d^1)$, and we have a direct sum decomposition
$$
E^1_{p, q} = \bigoplus_{\varphi \in \widehat{P}_q(X)}
C^\ell_{p+q}(\varphi).
$$
Given $\varphi = \langle \varphi_0, \cdots, \varphi_q \rangle \in \widehat{P}_q(X)$, a chain $\gamma \in P^\ell_{p+q}(\varphi)$ is expressed as
$$
\gamma =
\langle 
\varphi_0, x^1_1, \ldots, x^1_{n_1-1}, 
\varphi_1, \cdots,
\varphi_{q-1}, x^q_1, \ldots, x^q_{n_q-1},
\varphi_q
\rangle,
$$
where $n_1, \ldots, n_q$ are non-negative integers such that $n_1 + \cdots + n_q = p + q$, the points $x^j_i$ are smooth, and $\varphi_j$ are singular. (That $n_i = 0$ means there is no smooth point between $\varphi_{i-1}$ and $\varphi_i$.) Let $C^{\ell_1, \ldots, \ell_q}_{n_1, \ldots, n_q}(\varphi)$ be the subgroup in $C^\ell_{p+q}(\varphi)$ generated by such $(p+q)$-chains as above. The groups
$$
C^{\ell_1, \ldots, \ell_q}_{p+q}(\varphi)
= \bigoplus_{p+q = n_1 + \cdots + n_q}
C^{\ell_1, \ldots, \ell_q}_{n_1, \ldots, n_q}(\varphi)
$$
form a subcomplex $C^{\ell_1, \ldots, \ell_q}_{* + q}(\varphi)$ of $C^\ell_{* + q}(\varphi)$, so that we further get a direct sum decomposition
$$
E^1_{p, q}
= \bigoplus_{\varphi \in \widehat{P}_q(X)}
C^\ell_{p+q}(\varphi)
=
\bigoplus_{\varphi \in \widehat{P}_q(X)}
\bigoplus_{\stackrel{\ell_1, \cdots, \ell_q}{\ell_1 + \cdots + \ell_q = \ell}}
C^{\ell_1, \ldots, \ell_q}_{p+q}(\varphi).
$$

\medskip

Note that the spectral sequence preserves the direct sum decomposition of the magnitude homology $H^\ell_n(X)$ by the subgroups $H^\ell_n(a, b)$ with $a, b \in X$. We will write $E^r_{p, q}(a, b)$ for the $E^r$-term of the spectral sequence computing $H^\ell_n(a, b)$, which also provides a direct sum decomposition
$$
E^r_{p, q} = \bigoplus_{a, b \in X} E^r_{p, q}(a, b).
$$

\medskip

We now summarize the description of the magnitude homology $H^\ell_n(X)$ in terms of the spectral sequence for $n = 2, 3$. As mentioned above, it suffices to consider the spectral sequence computing the direct summand $H^\ell_2(a, b)$ for $a, b \in X$. Then $H^\ell_2(a, b)$ fits into the short exact sequence
$$
0 \to 
E^\infty_{0, 2}(a, b) \to 
H^\ell_2(a, b) \to 
E^\infty_{1,1}(a, b) \to 0.
$$
In addition, we have $E^\infty_{0, 2}(a, b) = E^3_{0, 2}(a, b)$ and $E^\infty_{1, 1}(a, b) = E^2_{1, 1}(a, b)$.  For $H^\ell_3(a, b)$, we have two exact sequences
\begin{gather*}
0 \to F_1H^\ell_3(a, b) \to H^\ell_3(a, b) \to E^\infty_{2, 1}(a, b) \to 0, \\
0 \to E^\infty_{0, 3}(a, b) \to F_1H^\ell_3(a, b) \to 
E^\infty_{1, 2}(a, b) \to 0,
\end{gather*}
where $F_1H^\ell_3(a, b) = \mathrm{Im}[H_3(F_1C_*^\ell(a, b), \partial) \to H^\ell_3(a, b)]$.


\subsection{Geodesic metric space}

In a metric space $(X, d)$, a \textit{geodesic} joining $x \in X$ to $y \in X$ is a map $f : [0, d(x, y)] \to X$ such that $f(0) = x$, $f(d(x, y)) = y$, and 
$$
\lvert t' - t \rvert = d(f(t'), f(t))
$$ 
for all $t, t' \in [0, d(x, y)]$. From this definition, the map $f$ is continuous. We say that $(X, d)$ is a \textit{geodesic (metric) space} if any pair of points $x, y \in X$ admits a geodesic joining $x$ to $y$. When such a geodesic is always unique, $(X, d)$ is called \textit{uniquely geodesic}. More details about geodesic spaces may be found in \cite{B-H,P} for example.

\medskip

We summarize here elementary properties of geodesic spaces to be used. 

\begin{lem} \label{lem:a_strong_Menger_convexity}
Let $(X, d)$ be a geodesic space. Then, for any distinct points $x, z \in X$ and any real number $\epsilon > 0$, there exists $y \in X$ such that $x < y < z$ and $d(x, y) < \epsilon$. 
\end{lem}

\begin{proof}
Suppose that two distinct points $x, z \in X$ and a real number $\epsilon > 0$ are given. Let $f : [0, d(x, z)] \to X$ be a geodesic joining $x$ to $z$. For any $t \in [0, d(x, z)]$, the point $y = f(t)$ satisfies
$$
d(x, y) + d(y, z) = 
\lvert 0 - t \rvert + \lvert t - d(x, z) \rvert
=
d(x, z).
$$
In particular, if $t \neq 0, d(x, z)$, then $x < y < z$. If we choose $t$ such that $0 < t < \epsilon$, then $d(x, y) = d(f(0), f(t)) = \lvert 0 - t \rvert = t < \epsilon$.
\end{proof}

It should be noticed that the conclusion in the lemma above is used in \cite{G} to ensure the vanishing of the smoothness spectral sequence $E^2_{0, q} = 0$ for $q \ge 2$. 

\smallskip

As seen in the proof above, if $f : [0, d(x, z)] \to X$ is a geodesic on a metric space $(X, d)$ joining $x$ to $z \neq x$, then we have $x < f(t) < z$ for any $t \in (0, d(x, z))$. A claim converse to this fact can be shown on a geodesic space:

\begin{lem} \label{lem:geodesic_from_x_y_z}
Let $(X, d)$ be a geodesic space. For any three points $x, y, z \in X$ such that $x < y < z$, there exists a geodesic $h : [0, d(x, z)] \to X$ which joins $x$ to $z$ and passes through $y$, namely, there is $t \in (0, d(x, z))$ such that $h(t) = y$.
\end{lem}

\begin{proof}
We begin with a property of a general metric space: Let $x, y, z \in X$ be such that $x < y < z$, so that $d(x, y) + d(y, z) = d(x, z)$. If $x' \in X$ and $z' \in X$ satisfy
\begin{align*}
d(x, x') + d(x', y) &= d(x, y), &
d(y, z') + d(z', z) &= d(y, z),
\end{align*}
then we have $d(x', y) + d(y, z') = d(x', z')$. This follows from
\begin{align*}
d(x', z') &= (d(x, x') + d(x', z') + d(z', z)) - (d(x, x') + d(z', z)) \\
&\ge
d(x, z) - d(x, x') - d(z', z) \\
&=
d(x, y) + d(y, z) - d(x, x') - d(z', z)
=
d(x', y) + d(y, z')
\end{align*}
and the triangle inequality $d(x', z') \le d(x', y) + d(y, z')$. 

Now, let $f : [0, d(x, y)] \to X$ be a geodesic joining $x$ to $y$, and $g : [0, d(y, z)] \to X$ a geodesic joining $y$ to $z$. We define a map $h : [0, d(x, z)] \to X$ by
$$
h(t) =
\left\{
\begin{array}{ll}
f(t), & (t \in [0, d(x, y)]) \\
g(t - d(x, y)). & (t \in [d(x, y), d(x, z)])
\end{array}
\right.
$$
Since $h(d(x, y)) = y$, the lemma will be completed by showing that $h$ is a geodesic joining $x$ to $z$. It is clear that $h(0) = f(0) = x$ and $h(d(x, z)) = g(d(y, z)) = z$. If $t, t' \in [0, d(x, y)]$ or $t, t' \in [d(x, y), d(x, z)]$, then $\lvert t' - t \rvert = d(h(t'), h(t))$ is also clear. If $t \in [0, d(x, y)]$ and $t' \in [d(x, y), d(x, z)]$, then we have
\begin{align*}
\lvert t - t' \rvert
&= \lvert t - d(x, y) \rvert  + \lvert - t' + d(x, y) \rvert \\
&= d(f(t), f(d(x, y))) + d(g(0), g(t' - d(x, y))) 
= d(x', y) + d(y, z'),
\end{align*}
where we put $x' = f(t)$ and $z' = g(t' - d(x, y))$. Notice that
\begin{align*}
d(x, y) 
&= 
\lvert 0 - t \rvert + \lvert t - d(x, y) \rvert
= d(x, x') + d(x', y), 
\\
d(y, z) 
&= 
\lvert 0 - t' + d(x, y) \rvert + \lvert t' - d(x, y) - d(y, z) \rvert
=
d(y, z') + d(z', z).
\end{align*}
Therefore we get
$$
\lvert t - t' \rvert
= d(x', y) + d(y, z')
= d(x', z')
= d(h(t), h(t')),
$$
and $h$ is a geodesic joining $x$ to $z$.
\end{proof}


\section{The second magnitude homology}
\label{sec:second_magnitude_homology}


\subsection{The case of general metric space}

Let $(X, d)$ be a metric space, $\ell$ a real number, and $a, b \in X$ two points. As is seen, the direct summand $H^\ell_2(a, b)$ of the magnitude homology $H^\ell_2(X)$ is trivial if $\ell < d(a, b)$, and is identified with the homology $\overline{H}_2(A(a, b))$ if $\ell = d(a, b)$ by Proposition \ref{prop:length_is_the_lower_bound}. Thus, we study the remaining case $\ell > d(a, b)$ by using the smoothness spectral sequence $E^r_{p, q}(a, b)$.

\begin{prop} \label{prop:E_11}
Let $(X, d)$ be a metric space, $\ell$ a real number, and $E_{p, q}^r(a, b)$ the smoothness spectral sequence for the direct summand $H^\ell_*(a, b)$ of the magnitude homology with $a, b \in X$. If $\ell > d(a, b)$, then $E^\infty_{1, 1}(a, b) = 0$. 
\end{prop}

\begin{proof}
Recall that $E^1_{p, q}(a, b)$ is identified with the free abelian group generated by the proper $(p+q)$-chains $\gamma \in P^\ell_{p+q}(a, b)$ such that $\sigma(\gamma) = p$. Thus, we have
\begin{align*}
E^1_{1, 1}(a, b) &=
\left\{
\begin{array}{ll}
\bigoplus_{a < x < b}
\Z \langle a, x, b \rangle, & (d(a, b) = \ell) \\
0. & (d(a, b) \neq \ell)
\end{array}
\right.
\end{align*}
As a result, if $d(a, b) \neq \ell$, then $E^\infty_{1, 1}(a, b) = 0$. 
\end{proof}

Next, we study $E^\infty_{0, 2}(a, b) = E^3_{0, 2}(a, b)$. For this aim, we recall:
\begin{align*}
E^1_{0, 2}(a, b) &= 
\bigoplus_{
\stackrel{a \neq \varphi \neq b, a \not< \varphi \not< b}
{\ell(\langle a, \varphi, b \rangle) = \ell}
} 
\Z \langle a, \varphi, b \rangle
=
\bigoplus_{\ell = d(a, \varphi) + d(\varphi, b) > d(a, b)}
\Z \langle a, \varphi, b \rangle, \\
E^1_{1, 2}(a, b) &=
\bigoplus_{
\stackrel{a < x < \varphi, x \not< \varphi \not< b, \varphi \neq b}
{\ell(\langle a, x, \varphi, b \rangle) = \ell}
}
\Z \langle a, x, \varphi, b \rangle 
\oplus
\bigoplus_{
\stackrel{a \neq \varphi, a \not< \varphi \not< x, \varphi < x < b}
{\ell(\langle a, \varphi, x, b \rangle) = \ell}
}
\Z \langle a, \varphi, x, b \rangle.
\end{align*}
We also recall the explicit description of $E^1_{2, 1}(a, b)$
$$
E^1_{2, 1}(a, b) =
\bigoplus_{
\stackrel{a < x_1 < x_2 < b}{\ell(\langle a, x_1, x_2, b \rangle) = \ell}
}
\Z \langle a, x_1, x_2, b \rangle.
$$

\begin{dfn} \label{dfn:simplicial_complex_B}
Let $(X, d)$ be a metric space, $\ell$ a real number, and $a, b \in X$ two points. In the case that $\ell > d(a, b)$, we define a $1$-dimensional simplicial complex $B^\ell(a, b)$ as follows:
\begin{enumerate}
\item[(a)]
A vertex ($0$-simplex) of $B^\ell(a, b)$ is a point $\varphi \in X$ such that:
\begin{enumerate}
\item[(i)]
$\ell = d(a, \varphi) + d(\varphi, b)$;

\item[(ii)]
there exists no point $x \in X$ such that $\langle a, x, \varphi, b \rangle \in P^\ell_3(\langle a, \varphi, b \rangle)$ or $\langle a, \varphi, x, b \rangle \in P^\ell_3(\langle a, \varphi, b \rangle)$; and

\item[(iii)]
there is no $\psi \in X$ such that: 
\begin{itemize}
\item
$\langle a, \varphi, \psi, b \rangle \in P^\ell_3(\langle a, b \rangle)$ and there exists no $x \in X$ such that $\langle a, x, \psi, b \rangle \in P^\ell_3(\langle a, \psi, b \rangle)$ or $\langle a, \psi, x, b \rangle \in P^\ell_3(\langle a, \psi, b \rangle)$; or 

\item
$\langle a, \psi, \varphi, b \rangle \in P^\ell_3(\langle a, b \rangle)$ and there exists no $x \in X$ such that $\langle a, x, \psi, b \rangle \in P^\ell_3(\langle a, \psi, b \rangle)$ or $\langle a, \psi, x, b \rangle \in P^\ell_3(\langle a, \psi, b \rangle)$.
\end{itemize}

\end{enumerate}

\item[(b)]
An edge ($1$-simplex) $\{ x, y \}$ of $B^\ell(a, b)$ consists of two vertices $\varphi$ and $\psi$ of $B^\ell(a, b)$ such that: $\langle a, \varphi, \psi, b \rangle \in P^\ell_3(\langle a, b \rangle)$; or $\langle a, \psi, \varphi, b \rangle \in P^\ell_3(\langle a, b \rangle)$.

\item[(c)]
We give each $1$-simplex $\{ \varphi, \psi \}$ an orientation by an order $\prec$ of the vertices: If $a < \varphi < \psi < b$, then $\varphi \prec \psi$. If $a < \psi < \varphi < b$, then $\psi \prec \varphi$. 
\end{enumerate}
\end{dfn}

We let $(C_*(B^\ell(a, b)), \partial)$ be the oriented chain complex of the simplicial complex $B^\ell(a, b)$, and $H_*(B^\ell(a, b))$ its homology.

\begin{lem} \label{lem:decomposition_of_E1_21}
Let $(X, d)$ be a metric space, $\ell$ a real number, and $E^r_{p, q}(a, b)$ the smoothness spectral sequence for the direct summand $H^\ell_*(a, b)$ of the magnitude homology with $a, b \in X$. We define ${}'E^1_{2, 1}(a, b), {}''E^1_{2, 1}(a, b) \subset E^1_{2, 1}(a, b)$ as follows.
\begin{itemize}
\item
${}'E^1_{2, 1}(a, b) \subset E^1_{2, 1}(a, b)$ is the subgroup spanned by $3$-chains $\langle a, x_1, x_2, b \rangle \in P^\ell_3(a, b)$ which are not $4$-cuts ($a < x_1 < b$ and $a < x_2 < b$), 

\item
${}''E^1_{2, 1}(a, b) \subset E^1_{2, 1}(a, b)$ is the subgroup spanned by $3$-chains $\langle a, x_1, x_2, b \rangle \in P^\ell_3(a, b)$ which are $4$-cuts ($a \not< x_1 \not< b$ and $a \not< x_2 \not< b$).
\end{itemize}
Then $E^1_{2, 1}(a, b)$ admits the direct sum decomposition
$$
E^1_{2, 1}(a, b) = {}'E^1_{2, 1}(a, b) \oplus {}''E^1_{2, 1}(a, b).
$$
\end{lem}

\begin{proof}
The verification is straightforward.
\end{proof}

\begin{lem} \label{lem:E_02}
Let $(X, d)$ be a metric space, $\ell$ a real number, and $E_{p, q}^r(a, b)$ the smoothness spectral sequence for the direct summand $H^\ell_*(a, b)$ of the magnitude homology with $a, b \in X$. Then $E^\infty_{0, 2}(a, b)$ has the expression
$$
E^\infty_{0, 2}(a, b) = 
E^1_{0, 2}(a, b)/(d^1(E^1_{1, 2}(a, b)) + \partial ({}''E^1_{2, 1}(a, b))),
$$
where $d^1(E^1_{1, 2}(a, b)) + \partial ({}''E^1_{2, 1}(a, b))$ is the subgroup consisting of the elements which are expressed as the sums of elements in $d^1(E^1_{1, 2}(a, b))$ and $\partial ({}''E^1_{2, 1}(a, b))$.
\end{lem}

\begin{proof}
As mentioned, $E^\infty_{0, 2}(a, b) = E^3_{0, 2}(a, b)$, and we have
\begin{align*}
E^3_{0, 2}(a, b) 
&= E^2_{0, 2}(a, b)/d^2(E^2_{2, 1}(a, b))
= (E^1_{0, 2}(a, b)/d^1(E^1_{1, 2}(a, b)))/d^2(E^2_{2, 1}(a, b)) \\
E^2_{2, 1}(a, b) 
&= Z^1_{2, 1}(a, b)/d^1(E^1_{3, 1}(a, b)),
\end{align*}
where we put $Z^1_{2, 1}(a, b) = \mathrm{Ker}[d^1 : E^1_{2, 1}(a, b) \to E^1_{1, 1}(a, b)]$. By definition, if an element $\bar{\gamma} \in E^2_{2, 1}(a, b)$ is represented by $\gamma \in Z^1_{2, 1}(a, b)$, then $d^2(\bar{\gamma})$ in $E^2_{0, 2}(a, b)$ is represented by $\partial \gamma \in E^1_{0, 2}(a, b)$. By the general nature of the spectral sequence, if $\gamma \in d^1(E^1_{3, 1}(a, b))$, then $\partial \gamma \in d^1(E^1_{1, 2}(a, b))$. This fact allows us to express
$$
E^3_{0, 2}(a, b) = 
E^1_{0, 2}(a, b)/(d^1(E^1_{1, 2}(a, b)) + \partial (Z^1_{2, 1}(a, b))).
$$
Applying Lemma \ref{lem:decomposition_of_E1_21} here, we uniquely decompose $\gamma \in Z^1_{2, 1}(a, b)$ as $\gamma = \gamma' + \gamma''$, where $\gamma' \in {}'E^1_{2, 1}(a, b)$ and $\gamma'' \in {}''E^1_{2, 1}(a, b)$. Since $d^1({}''E^1_{2, 1}(a, b)) = 0$ generally, we get $d^1\gamma' = d^1\gamma = 0$. We also have $d^1\gamma' = \partial \gamma'$ generally, so that $\partial \gamma = \partial \gamma''$. Consequently, we find $\partial (Z^1_{2, 1}(a, b)) = \partial ({}''E^1_{2, 1}(a, b))$, and the lemma is proved.
\end{proof}

\begin{prop} \label{prop:E_02}
Let $(X, d)$ be a metric space, $\ell$ a real number, and $a, b \in X$ two points. If $\ell > d(a, b)$, then we have an isomorphism of groups
$$
E^\infty_{0, 2}(a, b) \cong H_0(B^\ell(a, b)).
$$
\end{prop}

\begin{proof}
Let $\mathcal{P} = P^\ell_2(a, b)$ be the set of base elements of $E^1_{0, 2}(a, b)$. We introduce the following subsets $\mathcal{Q}, \mathcal{Q}' \subset \mathcal{P}$
\begin{align*}
\mathcal{Q}
&= \{ \langle a, \varphi, b \rangle |
\mbox{$\nexists x$ such that 
$\langle a, x, \varphi, b \rangle \in P^\ell_3(\langle a, \varphi, b \rangle)$ 
or
$\langle a, \varphi, x, b \rangle \in P^\ell_3(\langle a, \varphi, b \rangle)$}
\}, \\
\mathcal{Q}'
&= \{ \langle a, \varphi, b \rangle |
\mbox{$\exists x$ such that 
$\langle a, x, \varphi, b \rangle \in P^\ell_3(\langle a, \varphi, b \rangle)$ 
or
$\langle a, \varphi, x, b \rangle \in P^\ell_3(\langle a, \varphi, b \rangle)$}
\}.
\end{align*}
We also introduce subsets $\mathcal{R}, \mathcal{R}' \subset \mathcal{P}$
\begin{align*}
\mathcal{R}
&=
\bigg\{
\langle a, \varphi, b \rangle \bigg|\
\begin{array}{l}
\mbox{$\nexists \psi$ such that: 
$\langle a, \varphi, \psi, b \rangle \in P^\ell_3(\langle a, b \rangle)$ and 
$\langle a, \psi, b \rangle \in \mathcal{Q}'$; or} \\
\mbox{
$\langle a, \psi, \varphi, b \rangle \in P^\ell_3(\langle a, b \rangle)$ and 
$\langle a, \psi, b \rangle \in \mathcal{Q}'$} 
\end{array}
\bigg\}, \\
\mathcal{R}'
&=
\bigg\{
\langle a, \varphi, b \rangle \bigg|\
\begin{array}{l}
\mbox{$\exists \psi$ such that: 
$\langle a, \varphi, \psi, b \rangle \in P^\ell_3(\langle a, b \rangle)$ and 
$\langle a, \psi, b \rangle \in \mathcal{Q}'$; or} \\
\mbox{
$\langle a, \psi, \varphi, b \rangle \in P^\ell_3(\langle a, b \rangle)$ and 
$\langle a, \psi, b \rangle \in \mathcal{Q}'$} 
\end{array}
\bigg\}.
\end{align*}
We have $\mathcal{P} = \mathcal{Q} \sqcup \mathcal{Q}' = \mathcal{R} \sqcup \mathcal{R}'$. By definition, $E^1_{0, 2}(a, b) = \Z \mathcal{P}$. It is easy to see $d^1(E^1_{1, 2}(a, b)) = \Z \mathcal{Q}'$, so that 
$$
E^2_{0, 2}(a, b) 
= \Z \mathcal{P}/\Z \mathcal{Q}' 
= (\Z \mathcal{Q} \oplus \Z \mathcal{Q}')/\Z\mathcal{Q}'
\cong \Z \mathcal{Q}.
$$
Let $\tilde{\mathcal{S}} \subset \Z \mathcal{P}$ be the subset
$$
\tilde{\mathcal{S}}
= 
\{ \langle a, \varphi, b \rangle - \langle a, \psi, b \rangle |\
\mbox{$\langle a, \varphi, \psi, b \rangle \in P^\ell_3(\langle a, b \rangle)$
or $\langle a, \psi, \varphi, b \rangle \in P^\ell_3(\langle a, b \rangle)$}
\}.
$$
This subset is the disjoint union of $\mathcal{S}, \mathcal{S}', \mathcal{S}'' \subset \tilde{\mathcal{S}}$ give by
\begin{align*}
\mathcal{S}
&= 
\{ \langle a, \varphi, b \rangle - \langle a, \psi, b \rangle 
\in \tilde{\mathcal{S}} |\
\langle a, \varphi, b \rangle, \langle a, \psi, b \rangle \in \mathcal{Q}
\}, \\
\mathcal{S}'
&= 
\{ \langle a, \varphi, b \rangle - \langle a, \psi, b \rangle 
\in \tilde{\mathcal{S}} |\
\langle a, \varphi, b \rangle, \langle a, \psi, b \rangle \in \mathcal{Q}'
\},\\
\mathcal{S}''
&= 
\bigg\{ \langle a, \varphi, b \rangle - \langle a, \psi, b \rangle 
\in \tilde{\mathcal{S}} \bigg|\
\begin{array}{l}
\mbox{$\langle a, \varphi, b \rangle \in \mathcal{Q}$ and 
$\langle a, \psi, b \rangle \in \mathcal{Q}'$; or} \\
\mbox{$\langle a, \varphi, b \rangle \in \mathcal{Q}'$ and 
$\langle a, \psi, b \rangle \in \mathcal{Q}$}
\end{array}
\bigg\}.
\end{align*}
Note that $\partial({}''E^1_{2, 1}(a, b)) = \Z \tilde{\mathcal{S}} = \Z \mathcal{S} \oplus \Z \mathcal{S}' \oplus \Z \mathcal{S}''$. Note also $\Z \mathcal{S} \subset \Z \mathcal{Q}$ and $\Z\mathcal{S}' \subset \Z\mathcal{Q}'$. Now, we compute $E^\infty_{0, 2}(a, b)$ along Lemma \ref{lem:E_02} to get
$$
E^\infty_{0, 2}(a, b)
= \Z\mathcal{P}/(\Z \mathcal{Q}' 
+ \Z \mathcal{S} \oplus \Z \mathcal{S}' \oplus \Z \mathcal{S}'')
= \Z\mathcal{Q}/(\Z \mathcal{S} + \Z(\mathcal{Q} \cap \mathcal{R}')).
$$
Since $\mathcal{Q} = (\mathcal{Q} \cap \mathcal{R}) \sqcup (\mathcal{Q} \cap \mathcal{R}')$, we find $E^\infty_{0, 2}(a, b) = \Z(\mathcal{Q} \cap \mathcal{R})/\Z \mathcal{S}'''$, where $\mathcal{S}'''$ is the following subset of $\mathcal{S}$
$$
\mathcal{S}'''
= 
\{ \langle a, \varphi, b \rangle - \langle a, \psi, b \rangle 
\in \tilde{\mathcal{S}} |\
\langle a, \varphi, b \rangle, \langle a, \psi, b \rangle \in 
\mathcal{Q} \cap \mathcal{R}
\}.
$$
Recalling Definition \ref{dfn:simplicial_complex_B}, we can identify $\Z(\mathcal{Q} \cap \mathcal{R})$ with the free abelian group $C_0(B^\ell(a, b))$ generated by the vertices in $B^\ell(a, b)$, and $\Z \mathcal{S}'''$ with the image of the boundary map $\partial(C_1(B^\ell(a, b)))$, and hence $E^\infty_{0, 2}(a, b)$ with $H_0(B^\ell(a, b))$.
\end{proof}

We summarize the results so far to get a general description of the direct summand of the second magnitude homology.

\begin{thm} \label{thm:main_description}
Let $(X, d)$ be a metric space, $\ell$ a real number, and $a, b \in X$ two points. Then, by using the homology groups of the simplicial complexes $A(a, b)$ and $B^\ell(a, b)$, we can express the direct summand $H^\ell_2(a, b)$ of the second magnitude homology $H^\ell_2(X)$ as follows 
$$
H^\ell_2(a, b) \cong 
\left\{
\begin{array}{ll}
0, & (\ell < d(a, b)) \\
\overline{H}_0(A(a, b)), & (\ell = d(a, b)) \\
H_0(B^\ell(a, b)). & (\ell > d(a, b))
\end{array}
\right.
$$
\end{thm}

\begin{proof}
As is seen already, $H^\ell_2(a, b) = 0$ for $\ell < d(a, b)$, and $H^\ell_2(a, b) \cong \overline{H}_0(a, b)$ for $\ell = d(a, b)$ by Proposition \ref{prop:length_is_the_lower_bound}. Thus, we assume $\ell > d(a, b)$. Recall then the short exact sequence
$$
0 \to E^\infty_{0, 2}(a, b) \to H^\ell_2(a, b) \to E^\infty_{1, 1}(a, b) \to 0.
$$
By Proposition \ref{prop:E_11}, we have $E^\infty_{1, 1}(a, b) = 0$. Hence the theorem follows from the isomorphism $E^\infty_{0, 2}(a, b) \cong H_0(B^\ell(a, b))$ in Proposition \ref{prop:E_02}. 
\end{proof}

\begin{cor} \label{cor:main_descripton:torsion_free}
For any metric space $(X, d)$ and any real number $\ell$, the second magnitude homology group $H^\ell_2(X)$ is torsion free.
\end{cor}

\begin{cor} \label{cor:main_description:menger_convex}
Let $(X, d)$ be a metric space, and $\ell$ a positive number. Suppose that for any distinct points $x, z \in X$ and any positive number $\epsilon$ there exists $y \in X$ such that $x < y < z$ and $d(x, y) < \epsilon$. Then, for any $a, b \in X$, we have an isomorphism 
$$
H^\ell_2(a, b) \cong 
\left\{
\begin{array}{ll}
\overline{H}_0(A(a, b)), & (\ell = d(a, b)) \\
0. & (\ell \neq d(a, b))
\end{array}
\right.
$$
\end{cor}

\begin{proof}
As shown in \cite{G}, the assumption in the corollary implies $E^\infty_{0, 2} = 0$.
\end{proof}

By Lemma \ref{lem:a_strong_Menger_convexity}, the assumption in Corollary \ref{cor:main_description:menger_convex} is satisfied in the case that $(X, d)$ is a geodesic metric space.

\medskip

\begin{rem}
In \cite{L-S} (Theorem 7.22), some sufficient conditions for the vanishing of $H^\ell_2(X)$ are given. Concretely, if a metric space $X$ is geodetic and either:
\begin{itemize}
\item[(a)]
$X$ is Menger convex and has no $4$-cuts; or

\item[(b)]
$X$ is geodesic,
\end{itemize}
then $H^\ell_2(X) = 0$ for any $\ell$. In the context of the present paper, this result can be seen as follows: In general, the geodeticy of $X$ and the absence of $4$-cuts imply $E^\infty_{1, 1} = 0$. Then, each of (a) and (b) implies $E^\infty_{0, 2} = 0$. In the case of (a), the proof in \cite{L-S} is actually showing $E^\infty_{0, 2} = 0$. In the case of (b), the assumption in Corollary \ref{cor:main_description:menger_convex} is satisfied, and $E^\infty_{0, 2} = 0$.
\end{rem}

\begin{rem}
In \cite{L-S} (Theorem 7.25), a description of $H^\ell_2(X)$ in terms of a free abelian group is given, under the assumption that $X$ is geodetic and has no $4$-cuts. This result can also be seen in the context of the present paper as follows: As is remarked, the geodeticy and the absence of $4$-cuts imply $E^\infty_{1, 1} = 0$. The absence of $4$-cuts then implies the $E^2$-degeneracy of the spectral sequence \cite{G} (Theorem 3.9). The absence of $4$-cuts also simplifies $B^\ell(a, b)$, and we get
$$
H^\ell_2(X) = E^2_{0, 2} = \bigoplus_{a, b \in X} C_0(B^\ell(a, b)),
$$
which reproduces the description in \cite{L-S}.
\end{rem}


\subsection{The case of geodesic metric space}
\label{subsec:second_homology_on_geodesic_space}

Let $(X, d)$ be a geodesic metric space, $\ell$ a real number, and $H^\ell_2(a, b)$ the direct summand of the second magnitude homology $H^\ell_2(X)$ for $a, b \in X$. Then, by Lemma \ref{lem:a_strong_Menger_convexity} and Corollary \ref{cor:main_description:menger_convex}, the group $H^\ell_2(a, b)$ can be non-trivial only when $\ell = d(a, b)$, and is identified with, in this case, the reduced homology $\overline{H}_0(A(a, b))$ of the simplicial complex $A(a, b)$. We here interpret $\overline{H}_0(A(a, b))$ by means of certain equivalence classes of geodesics. 

\smallskip

For this aim, we apply a description of the zeroth homology of a simplicial complex (see \cite{Spa}): In general, a simplicial complex induces a groupoid, i.e.\ a category whose morphisms are invertible. In the groupoid induced from $A(a, b)$, an object is a $0$-simplex in $A(a, b)$, namely $x \in X$ such that $a < x < b$. For each object $x$, there is the identity morphism. The other morphisms are specified by a $1$-simplex, namely, there are morphisms $x \to y$ and $y \to x$ inverse to each other if $a < x < y < b$. Finally a $2$-simplex specifies a composition of morphisms: If $a < x < y < z < b$, then the morphisms $x \to y$ and $y \to z$ compose to give $x \to z$. Now, let $\pi_0(A(a, b))$ denote the set of isomorphism classes in the groupoid induced from the simplicial complex $A(a, b)$. Then the zeroth unreduced homology group of $A(a, b)$ is isomorphic to the free abelian group generated by the isomorphism classes in $\pi_0(A(a, b))$
$$
H_0(A(a, b)) \cong \Z[\pi_0(A(a, b))] 
= \bigoplus_{[x] \in \pi_0(A(a, b))} \Z[x].
$$
We have a surjective homomorphism $\epsilon : \Z[\pi_0(A(a, b))] \to \Z$ given by $\epsilon([x]) = 1$ for each isomorphism class $[x]$ in the groupoid. Its kernel is isomorphic to the zeroth reduced homology of $A(a, b)$
$$
\overline{H}_0(A(a, b))
\cong
\mathrm{Ker}[\epsilon : \Z[\pi_0(A(a, b))] \to \Z].
$$

\begin{dfn} \label{dfn:equivalence_geodesic}
Let $(X, d)$ be a metric space. 
\begin{itemize}
\item
For $a, b \in X$, we let $\Geod(a, b)$ be the set of geodesics joining $a$ to $b$. 

\item
We introduce a relation $\sim$ to $\Geod(a, b)$ as follows: Let $f : [0, d(a, b)] \to X$ and $g : [0, d(a, b)] \to X$ be geodesics joining $a$ to $b$. Then $f \sim g$ if and only if they intersect at a point other than $a$ or $b$. Concretely, there exists $t \in (0, d(a, b))$ such that $f(t) = g(t)$. This is equivalent to 
$$
f((0, d(a, b))) \cap g((0, d(a, b))) \neq \emptyset.
$$

\item
We define $\approx$ to be the equivalence relation generated by $\sim$. 

\item
We write $\pi_0(\Geod(a, b)) = \Geod(a, b)/\!\!\approx$ for the set of equivalence classes.
\end{itemize}
\end{dfn}

We can introduce a groupoid such that the set of its isomorphism classes is exactly $\Geod(a, b)/\!\!\approx$. The detailed formulation is left to interested readers.

\begin{lem} \label{lem:map_Phi}
Let $(X, d)$ be a geodesic space. For distinct points $a, b \in X$, there is a well-defined map
$$
\Phi : \ 
\pi_0(A(a, b)) \to \pi_0(\Geod(a, b)).
$$
\end{lem}

\begin{proof}
We construct $\Phi$ as follows: Let $x$ be an object in the groupoid induced from the simplicial complex $A(a, b)$. This means that $x$ is a $0$-simplex in $A(a, b)$, namely, a point $x \in X$ such that $a < x < b$. By Lemma \ref{lem:geodesic_from_x_y_z}, there is a geodesic $f_x : [0, d(a, b)] \to X$ joining $a$ to $b$ and passing through $x$. We then put $\Phi([x]) = [f_x]$. 

This $\Phi$ is well-defined. To see this, let $x$ and $y$ be distinct isomorphic objects in the groupoid induced from $A(a, b)$. Hence $x, y \in X$ satisfy $a < x < y < b$ or $a < y < x < b$. By a simple generalization of the proof of Lemma \ref{lem:geodesic_from_x_y_z}, we can construct a geodesic $g : [0, d(a, b)] \to X$ joining $a$ to $b$ and passing through $x$ and $y$. Thus, in the (equivalence) relation of $\Geod(a, b)$, we have $f_x \sim g \sim f_y$. This proves that $\Phi$ is well-defined.
\end{proof}

\begin{lem}  \label{lem:map_Psi}
Let $(X, d)$ be a geodesic space. For distinct points $a, b \in X$, there is a well-defined map
$$
\Psi : \ \pi_0(\Geod(a, b)) \to \pi_0(A(a, b)).
$$
\end{lem}

\begin{proof}
The construction of $\Psi$ is as follows: Let $f : [0, d(a, b)] \to X$ be a geodesic joining $a$ to $b$. If we choose $t \in (0, d(a, b))$, then $a \neq f(t) \neq b$ and
$$
d(a, b) = \lvert 0 - t \rvert + \lvert t - d(a, b) \rvert
= d(a, f(t)) + d(f(t), b).
$$
That is, $a < f(t) < b$. Now, we put $\Psi([f]) = [f(t)]$. 

To prove that this $\Psi$ is well-defined, let $t' \in (0, d(a, b))$ be another choice. Without loss of generality, we can assume $t < t'$. Then we see
\begin{align*}
d(a, f(t')) &= \lvert 0 - t \rvert + \lvert t - t' \rvert
= d(a, f(t)) + d(f(t), f(t')), \\
d(f(t), b) &= \lvert t - d(a, b) \rvert
= \lvert t - t' \rvert + \lvert t' - d(a, b) \rvert
= d(f(t), f(t')) + d(f(t'), b).
\end{align*}
It follows that $a < f(t) < f(t') < b$, so that $f(t)$ and $f(t')$ are isomorphic as objects in the groupoid induced from $A(a, b)$. To complete the proof, let $g : [0, d(a, b)] \to X$ be another geodesic joining $a$ to $b$ such that $f \sim g$. This means that $f(t_0) = g(t_0)$ for some $t_0 \in (0, d(a, b))$. Therefore $\Psi([f]) = [f(t_0)] = [g(t_0)] = \Psi([g])$, and the map $\Psi$ is well-defined. 
\end{proof}

\begin{thm}
Let $(X, d)$ be a geodesic space. For distinct points $a, b \in X$, there is a bijection
$$
\pi_0(A(a, b)) \cong \pi_0(\Geod(a, b)).
$$
Consequently, there is an isomorphism of groups
$$
H_0(A(a, b)) \cong \Z[\pi_0(\Geod(a, b))].
$$
\end{thm}

\begin{proof}
The map $\Phi$ in Lemma \ref{lem:map_Phi} and $\Psi$ in Lemma \ref{lem:map_Psi} are inverse to each other. The remaining claim follows from $H_0(A(a, b)) \cong \Z[\pi_0(A(a, b))]$.
\end{proof}

Under the bijection in the theorem, the natural augmentation $\epsilon : H_0(A(a, b)) \to \Z$ corresponds to the homomorphism $\epsilon : \Z[\pi_0(\Geod(a, b))] \to \Z$ that carries the equivalence class $[f]$ of each geodesic $f \in \Geod(a, b)$ to $\epsilon([f]) = 1$. Using this induced augmentation, we can finally describe the direct summand $H^\ell_2(a, b)$ in terms of geodesics.

\begin{cor} \label{cor:second_homology_on_geodesic_space}
Let $(X, d)$ be a geodesic space, and $\ell$ a real number. For $a, b \in X$, we can describe the direct summand $H^\ell_2(a, b)$ of the second magnitude homology as follows 
$$
H^\ell_2(a, b) \cong
\left\{
\begin{array}{ll}
\mathrm{Ker}[\epsilon : \Z[\pi_0(\Geod(a, b))] \to \Z], & (d(a, b) = \ell) \\
0. & (d(a, b) \neq \ell)
\end{array}
\right.
$$
\end{cor}

If we choose a geodesic $\overline{f} \in \Geod(a, b)$ as a reference, then we can express
$$
H^{d(a, b)}_2(a, b) \cong
\bigoplus_{[f] \neq [\overline{f}]} \Gamma([f], [\overline{f}]),
$$
where $[f] \in \pi_0(\Geod(a, b))$ runs over geodesic classes such that $[f] \neq [\overline{f}]$, and $\Gamma([f], [\overline{f}]) = [f] - [\overline{f}] \in \Z[\pi_0(\Geod(a, b))]$. In the magnitude homology, $\Gamma([f], [\overline{f}])$ is represented by $\langle a, f(t), b \rangle - \langle a, \overline{f}(t), b \rangle \in C^{d(a, b)}_2(a, b)$, where $t \in (0, d(a, b))$.


\section{The third magnitude homology}
\label{sec:third_magnitude_homology}


\subsection{The case of geodesic metric space}

On a geodesic space $(X, d)$, the direct summand $H^\ell_2(a, b)$ of the second magnitude homology can be non-trivial only when $\ell = d(a, b)$, and this case admits a description by using $A(a, b)$ as given in Proposition \ref{prop:length_is_the_lower_bound}. The purpose of this subsection is to generalize this result to the direct summand $H^\ell_3(a, b)$ of the third magnitude homology. For its achievement, we will compute the smoothness spectral sequence with the aid of some lemmas.

\begin{lem} \label{lem:first_preliminary_to_kill_4_cut}
Let $(X, d)$ be a geodesic space. If we are given points $a, x, y, b \in X$ such that $a < x < y < b$, $a \not< y \not< b$ and $a \not< x \not< b$ (namely $\langle a, x, y, b \rangle$ is a $4$-cut), then there exists $z \in X$ such that $x < y < z < b$ and $a \not< y \not< z$.
\end{lem}

\begin{proof}
Since $x < y < b$, we have by Lemma \ref{lem:geodesic_from_x_y_z} a geodesic $f : [0, d(x, b)] \to X$ joining $x$ to $b$ and passing through $y$. We then consider a continuous function 
$$
F(t) = d(a, y) + d(y, f(t)) - d(a, f(t))
$$
defined on $d(x, y) \le t \le d(x, b)$. We have $F(d(x, b)) > 0$ because $a \not< y \not< b$. Thus, there is $t_0$ near $d(x, b)$ at which $F(t_0) > 0$. Put differently, there is $z = f(t_0)$ such that $a \not< y \not< z$. Since $z$ is on the geodesic $f$, it holds that $x < y < z < b$.
\end{proof}

\begin{lem} \label{lem:second_preliminary_to_kill_4_cut}
Let $(X, d)$ be a metric space. If we have 
\begin{itemize}
\item
$a, x, y, b \in X$ such that $a < x < y < b$; and 

\item
$z \in X$ such that $x < y < z < b$ and $a \not< y \not< z$,
\end{itemize}
then $a \not< x \not< z$ and $x < z < b$.
\end{lem}

\begin{proof}
From $x < y < z$, $a < x < y$ and $a \not< y \not< z$, we get
\begin{align*}
d(a, x) + d(x, z) - d(a, z)
&= d(a, x) + d(x, y) + d(y, z) - d(a, z) \\
&= d(a, y) + d(y, z) - d(a, z) > 0,
\end{align*}
namely, $a \not< x \not< z$. From $x < y < z < b$ and $x < y < b$, we have
$$
d(x, z) + d(z, b) 
= d(x, y) + d(y, z) + d(z, b)
= d(x, y) + d(y, b)
= d(x, b),
$$
so that $x < z < b$. 
\end{proof}

\begin{prop} \label{prop:vanishing_E2_21_ell_general}
Let $(X, d)$ be a geodesic space, $\ell$ a real number, and $E^r_{p, q}(a, b)$ the smoothness spectral sequence for the direct summand $H^\ell_*(a, b)$ of the magnitude homology with $a, b \in X$. If $\ell > d(a, b)$, then $E^2_{2, 1}(a, b) = 0$.
\end{prop}

\begin{proof}
If $\ell > d(a, b)$, then $E^1_{1, 1}(a, b) = 0$ and $E^2_{2, 1}(a, b) = E^1_{2, 1}(a, b)/d^1(E^1_{3, 1}(a, b))$. We have the expression
$$
E^1_{2, 1}(a, b) =
\bigoplus_{\stackrel{a < x < y < b}{\ell(\langle a, x, y, b \rangle) = \ell}}
\Z \langle a, x, y, b \rangle.
$$
Let $\langle a, x, y, b \rangle$ be a base element of $E^1_{2, 1}(a, b)$. From $\ell > d(a, b)$, it follows that $a < x < y < b$, $a \not< y \not< b$ and $a \not< x \not< b$, namely $\langle a, x, y, b \rangle$ is a $4$-cut. Thus, we have a point $z \in X$ such as in Lemma \ref{lem:first_preliminary_to_kill_4_cut}. From $x < y < z < b$, we get a chain $\langle a, x, y, z, b \rangle \in E^1_{3, 1}(a, b)$. We have $a \not< y \not< z$ and $y < z < b$. Hence the chain $\langle a, y, z, b \rangle $ does not belong to $E^1_{2, 1}(a, b)$. By Lemma \ref{lem:second_preliminary_to_kill_4_cut}, we have $a \not< x \not< z$ and $x < z < b$. Thus $\langle a, x, z, b \rangle$ does not belong to $E^1_{2, 1}(a, b)$. As a result, we get
$$
d^1\langle a, x, y, z, b \rangle
= - \langle a, x, y, b \rangle.
$$
This implies $E^1_{2, 1}(a, b) = d^1(E^1_{3, 1}(a, b))$ and $E^2_{2, 1}(a, b) = 0$. 
\end{proof}

\begin{lem} \label{lem:preliminary_to_the_vanishing_E2_12}
Let $(X, d)$ be a geodesic space. Suppose that we have
\begin{itemize}
\item
$a, \varphi, b \in X$ such that $a \neq \varphi \neq b$ and $a \not< \varphi \not< b$,

\item
$x_i \in X$ such that $a < x_i < \varphi$ and $x_i \not< \varphi \not< b$ for $i = 1, \cdots, m$,

\item
$y_j \in X$ such that $a \not< \varphi \not< y_j$ and $\varphi < y_j < b$ for $j = 1, \cdots, n$,
\end{itemize}
Then there exist $\overline{x}, \overline{y} \in X$ such that 
\begin{align*}
&a < \overline{x} < \varphi, &
&\overline{x} \not< \varphi \not< b, &
&\overline{x} \not< \varphi \not< y_j, \quad (j = 1, \cdots, n) \\
&a \not< \varphi \not< \overline{y}, &
&\varphi < \overline{y} < b, &
&x_i \not< \varphi \not< \overline{y}, \quad (i = 1, \cdots, m) \\
&\overline{x} \not< \varphi \not< \overline{y}.
\end{align*}
\end{lem}

\begin{proof}
We define continuous functions in $x, y \in X$ as follows:
\begin{align*}
F_j(x) &= 
\left\{
\begin{array}{ll}
d(x, \varphi) + d(\varphi, b) - d(x, b), & (j = 0) \\
d(x, \varphi) + d(\varphi, y_j) - d(x, y_j), & (j = 1, \cdots, n) 
\end{array}
\right.
\\
G_i(x) &=
\left\{
\begin{array}{ll}
d(a, \varphi) + d(\varphi, y) - d(a, y), & (i = 0) \\
d(x_i, \varphi) + d(\varphi, y) - d(x_i, y), & (i = 1, \cdots, m) 
\end{array}
\right. 
\\
H(x, y) &= d(x, \varphi) + d(\varphi, y) - d(x, y).
\end{align*}
Since $a \not< \varphi \not< b$ and $a \not< \varphi \not< y_j$, we have $F_j(a) > 0$. Similarly $a \not< \varphi \not< b$ and $x_i \not< \varphi \not< b$ imply $G_i(b) > 0$. From $a \not< \varphi \not< b$, it also follows that $H(a, b) > 0$. As the product of these functions, we define a continuous function
$$
I(x, y) = 
\bigg(\prod_{j = 0}^n F_j(x)\bigg)
\bigg(\prod_{i = 0}^m G_i(x)\bigg)
H(x, y).
$$
Now, let $f : [0, d(a, \varphi)] \to X$ be a geodesic joining $a$ to $\varphi$, and $g : [0, d(b, \varphi)] \to X$ a geodesic joining $b$ to $\varphi$. The composite function $I(f(t), g(u))$ is continuous in $t$ and $u$, and $I(0, 0) > 0$. Hence there are $t_0 \in (0, d(a, \varphi))$ and $u_0 \in (0, d(b, \varphi))$ such that $I(f(t_0), g(u_0)) > 0$. Now, $\overline{x} = f(t_0)$ and $\overline{y} = g(t_0)$ have the required properties.
\end{proof}

\begin{prop} \label{prop:vanishing_E2_12}
Let $(X, d)$ be a geodesic space, $\ell$ a real number, and $E^r_{p, q}(a, b)$ the smoothness spectral sequence for the direct summand $H^\ell_*(a, b)$ of the magnitude homology with $a, b \in X$. Then $E^2_{1, 2}(a, b) = 0$.
\end{prop}

\begin{proof}
We apply the direct sum decomposition of the chain complex
$$
E^1_{p, 2}(a, b) =
\bigoplus_{\langle a, \varphi, b \rangle \in \widehat{P}_q} C^\ell_{p+2}(
\langle a, \varphi, b \rangle),
$$
where $C^\ell_{p + 2}(\langle a, \varphi, b \rangle)$ is the free abelian group generated by $(p+2)$-chains 
$$
\langle a, x_1, \cdots, x_m, \varphi, y_1, \cdots, y_n, b \rangle
\in P^\ell_{p+2}(\langle a, \varphi, b \rangle)
$$ 
of length $\ell$ in which $x_1, \cdots, x_m$ and $y_1, \cdots, y_n$ are smooth, whereas $a, \varphi, b$ are singular. ($m$ and $n$ satisfy $m + n = p$.) If $a = \varphi$ or $\varphi = b$, then $C^\ell_3(\langle a, \varphi, b \rangle) = 0$. Actually, $C^\ell_3(\langle a, \varphi, b \rangle)$ is generated by proper $3$-chains of the forms
\begin{align*}
&\langle a, x, \varphi, b \rangle, &
&\langle a, \varphi, y, b \rangle,
\end{align*}
in which $x$ and $y$ are smooth. If $a = \varphi = b$, then there are no proper chains of the above forms. If $a = \varphi \neq b$, then $C^\ell_3(\langle a, \varphi, b \rangle)$ is generated by proper chains of the form $\langle a, x, \varphi, b \rangle$ in which $x$ is smooth. However, $x$ cannot be smooth, because $a = \varphi$. The same argument applies to the case that $a \neq \varphi = b$, and we find $C^\ell_3(\langle a, \varphi, b \rangle) = 0$ whenever $a = \varphi$ or $\varphi = b$.

Henceforth we assume $a \neq \varphi \neq b$. Under this assumption, if $d(a, \varphi) + d(\varphi, b) = d(a, b)$, then $C^\ell_3(\langle a, \varphi, b \rangle) = 0$. This is because there are no proper $3$-chains of the forms $\langle a, x, \varphi, b \rangle$ or $\langle a, \varphi, y, b \rangle$ in which $x, y$ are smooth but $\varphi$ is singular. Hence we also assume $d(a, \varphi) + d(\varphi, b) > d(a, b)$. Under these two assumptions, if $\ell \neq d(a, \varphi) + d(\varphi, b)$, then $C^\ell_3(\langle a, \varphi, b \rangle) = 0$. To summarize, we assume that $a \neq \varphi \neq b$ and $\ell = d(a, \varphi) + d(\varphi, b) > d(a, b)$ in the following. In this case, we can describe $C^\ell_{p + 2}(\langle a, \varphi, b \rangle)$ for $p = 0, 1$ as follows:
\begin{align*}
C^\ell_2(\langle a, \varphi, b \rangle)
&= 
\Z \langle a, \varphi, b \rangle, \\
C^\ell_3(\langle a, \varphi, b \rangle)
&=
\bigoplus_{\stackrel{a < x < \varphi}{x \not< \varphi \not< b}}
\Z \langle a, x, \varphi, b \rangle
\oplus
\bigoplus_{\stackrel{a \not< \varphi \not< y}{\varphi < y < b}}
\Z \langle a, \varphi, y, b \rangle.
\end{align*}

Now, given an element $\zeta \in C^\ell_3(\langle a, \varphi, b \rangle)$, we can express it as follows
$$
\zeta = \sum_{i = 1}^m k_i \langle a, x_i, \varphi, b \rangle
+ \sum_{j = 1}^n \ell_j \langle a, \varphi, y_j, b \rangle,
$$
where $x_i$ and $y_j$ are smooth. The thing we will prove is that the chains in $\zeta$ are homologous to each other. Once this is proved, the cycle condition $d^1\zeta = 0$ will imply that $\zeta$ is a boundary, so that $H^\ell_3(\langle a, \varphi, b \rangle) = 0$ and $E^2_{1, 2}(a, b) = 0$. Applying Lemma \ref{lem:preliminary_to_the_vanishing_E2_12} to the points $a, \varphi, b, x_i$ and $y_j$ in the chains constituting $\zeta$, we have $\overline{x}$ and $\overline{y}$. From the properties of $\overline{x}$ and $\overline{y}$, we have the following chains
$$
\langle a, x_i, \varphi, \overline{y}, b \rangle, \
\langle a, \overline{x}, \varphi, y_j, b \rangle, \
\langle a, \overline{x}, \varphi, \overline{y}, b \rangle \
\in C^\ell_4(\langle a, \varphi, b \rangle).
$$
Their boundaries are
\begin{align*}
d^1\langle a, x_i, \varphi, \overline{y}, b \rangle
&= - \langle a, \varphi, \overline{y}, b \rangle 
- \langle a, x_i, \varphi, b \rangle, \\
d^1\langle a, \overline{x}, \varphi, y_j, b \rangle
&= - \langle a, \varphi, y_j, b \rangle
- \langle a, \overline{x}, \varphi, b \rangle, \\
d^1\langle a, \overline{x}, \varphi, \overline{y}, b \rangle
&= - \langle a, \varphi, \overline{y}, b \rangle
- \langle a, \overline{x}, \varphi, b \rangle. 
\end{align*}
Hence all the chains in $\zeta$ are homologous to $\langle a, \varphi, \overline{y}, b \rangle$.
\end{proof}

Based on the vanishing results so far, we get the following description of the third magnitude homology of a geodesic space.

\begin{thm} \label{thm:third_homology_on_geodesic_space}
Let $(X, d)$ be a geodesic space, and $\ell$ a real number. For any $a, b \in X$, there is an isomorphism 
$$
H^\ell_3(a, b) \cong
\left\{
\begin{array}{ll}
H_1(A(a, b)), & (\ell = d(a, b)) \\
0. & (\ell \neq d(a, b))
\end{array}
\right.
$$
\end{thm}

\begin{proof}
It is clear that $H^\ell_3(a, b) = 0$ for $\ell < d(a, b)$. The case of $\ell = d(a, b)$ is already seen in  Proposition \ref{prop:length_is_the_lower_bound}. Therefore we investigate the case of $\ell > d(a, b)$ by using the smoothness spectral sequence. Notice that $E^1_{p, 0}(a, b) = 0$ for all $p$ and $\ell > 0$ by construction. Notice also that $E^2_{0, q}(a, b) = 0$ for $q \ge 2$ by a vanishing result in \cite{G} and Lemma \ref{lem:a_strong_Menger_convexity}. Thus, by $E^2_{0, 3}(a, b) = 0$ together with $E^2_{1, 2}(a, b) = 0$ from Proposition \ref{prop:vanishing_E2_12}, we get $H^\ell_3(a, b) \cong E^\infty_{2, 1}(a, b)$. This is trivial, because $E^2_{2, 1}(a, b) = 0$ if $\ell > d(a, b)$ as established in Proposition \ref{prop:vanishing_E2_21_ell_general}.
\end{proof}

\begin{cor} \label{cor:third_homology_on_geodesic_space_simple_version}
Let $(X, d)$ be a geodesic space, and $a, b \in X$ distinct points. If there is only one geodesic joining $a$ to $b$, then $H^\ell_3(a, b) = 0$ for any $\ell$.
\end{cor}

\begin{proof}
It suffices to show $H_1(A(a, b)) = \overline{H}_1(A(a, b)) = 0$ under the assumption $\ell = d(a, b)$. By the hypothesis and a simple generalization of Lemma \ref{lem:geodesic_from_x_y_z}, all the vertices of any $(p-1)$-simplex $\{ x_1, \cdots, x_p \}$ in $A(a, b)$ lie on the unique geodesic $f$ joining $a$ to $b$. This fact implies that the vertices in $A(a, b)$ are totally ordered. Then we can apply Lemma \ref{lem:vanishing_of_homology_in_totally_ordered_case}.
\end{proof}

To have a generalization of the above corollary, we introduce an assumption, which concerns with a certain non-branching property of a geodesic space.

\begin{assumption} \label{assumption:non_branching}
Let $(X, d)$ be a geodesic space. We assume that, for any $a, b \in X$, if two geodesics $f$ and $g$ joining $a$ to $b$ share a point other than $a$ or $b$ (namely, there exists $t \in (0, d(a, b))$ such that $f(t) = g(t)$), then $f = g$. 
\end{assumption}

Note that the assumption is equivalent to that, for any $a, b \in X$, each equivalence class in $\pi_0(\Geod(a, b))$ is represented by a unique geodesic, and hence we get $\pi_0(\Geod(a, b)) = \Geod(a, b)$. Note also that any connected and complete Riemannian manifold, which is a geodesic space \cite{B-H,P}, fulfils the assumption. This is because a geodesic in this case is locally characterized by an ordinary differential equation, so the uniqueness in the assumption follows from that of a solution to the ordinary differential equation with an initial condition.

\begin{cor} \label{cor:third_homology_on_geodesic_space_general_version}
Let $(X, d)$ be a geodesic space, and $a, b \in X$ distinct points. If Assumption \ref{assumption:non_branching} is satisfied, then $H^\ell_3(a, b) = 0$ for any $\ell$.
\end{cor}

This vanishing result will be generalized in \S\ref{sec:higher_magnitude_homology}. But, it would be instructive to give a proof in this particular case, since it contains a prototypical argument.

\begin{proof}
Again, it suffices to show $H_1(A(a, b)) = 0$ under $\ell = d(a, b)$. In general, the homology group $H_n(A(a, b))$ admits a direct sum decomposition by means of the ``path connected component''. To be precise, notice that all the vertices of a $(p-1)$-simplex $\{ x_1, \cdots, x_p \}$ in $A(a, b)$ lie on a geodesic $f$ joining $a$ to $b$. This is a consequence of a simple generalization of Lemma \ref{lem:geodesic_from_x_y_z}. Generally, such a geodesic $f$ is not unique, but its equivalence class $[f] \in \pi_0(\Geod(a, b))$ is uniquely determined by the simplex. Thus, in this case, we say that the geodesic class of the simplex is $[f]$. Let $A(a, b; \xi) \subset A(a, b)$ be the subcomplex consisting of simplices whose geodesic classes are $\xi \in \pi_0(A(a, b)) \cong \pi_0(\Geod(a, b))$. Then the simplicial complex $A(a, b)$ is expressed as the disjoint union 
$$
A(a, b) = \bigsqcup_{\xi \in \pi_0(A(a, b))} A(a, b; \xi),
$$
which leads to the direct sum decomposition of the homology group
$$
H_n(A(a, b)) = \bigoplus_{\xi \in \pi_0(\Geod(a, b))} H_n(A(a, b; \xi)).
$$
Now, by the hypothesis of the corollary, each equivalence class $\xi \in \pi_0(\Geod(a, b))$ is represented by a unique geodesic $f$. This means that all the vertices in $A(a, b; \xi)$ lie on $f$. Now, we can apply the argument in the proof of Corollary \ref{cor:third_homology_on_geodesic_space_simple_version} to showing $H_n(A(a, b; \xi)) = 0$ for $n \ge 1$. 
\end{proof}


\subsection{An invariant of the third magnitude homology}
\label{subsec:invariant}

Let $(X, d)$ be a metric space, and $a, b \in X$ two points. If $\ell = d(a, b)$, then $H^\ell_3(a, b)$ is isomorphic to the first homology group of the simplicial complex $A(a, b)$ by Proposition \ref{prop:length_is_the_lower_bound}. In general, a first homology class of a simplicial complex can be represented by a closed \textit{edge path} \cite{Spa}. In the present setup, each edge of an edge path representing an element in $H^\ell_3(a, b)$ corresponds to a $3$-chain $\langle a, x, y, b \rangle$. Here, if $X$ is a geodesic space, then there is a geodesic joining $x$ to $y$. Then, an edge path can be thought of as a geometric path in $X$ consisting of segments which lie on geodesics joining $a$ to $b$. The idea to formulate our invariant of third magnitude homology classes is to consider the intersection number of the geometric (edge) path and a geodesic joining $a$ to $b$. Notice, however, that this naive idea seems not to work when geodesics joining $x$ and $y$ are not unique. This is the reason that we put the following assumption.

\begin{assumption} \label{assumption:for_intersection_number}
Let $(X, d)$ be a metric space, $a, b \in X$ two points, and $\ell = d(a, b) > 0$. We assume that there is a geodesic $f : [0, d(a, b)] \to X$ joining $a$ to $b$. We also assume that, for any $x, y \in X$ such that $a < x < y < b$, there is a unique geodesic $g : [0, d(x, y)] \to X$ joining $x$ to $y$. 
\end{assumption}

This assumption leads to the following description of the geometric intersection of relevant geodesics.

\begin{lem} \label{lem:geometric_intersection}
Under Assumption \ref{assumption:for_intersection_number}, the inverse image $g^{-1}(\mathrm{Im}f) \subset [0, d(x, y)]$ is either the empty set, a point or a closed interval of positive length.
\end{lem}

\begin{proof}
Suppose $g^{-1}(\mathrm{Im}f) \neq \emptyset$. Being the image of a compact space, $\mathrm{Im}f$ is also compact. Hence $g^{-1}(\mathrm{Im}f) \subset [0, d(x, y)]$ is a bounded closed set, and we can make sense of the real numbers
\begin{align*}
t_{\mathrm{min}} 
&=
\min \{ t \in g^{-1}(\mathrm{Im}f) \}, 
&
t_{\mathrm{max}} 
&=
\max \{ t \in g^{-1}(\mathrm{Im}f) \}.
\end{align*}
If $t_{\mathrm{min}} = t_{\mathrm{max}}$, then $g^{-1}(\mathrm{Im}f)$ consists of a point. Suppose $t_{\mathrm{min}} < t_{\mathrm{max}}$. The proof of the lemma will be completed once we could show that $g(t) \in \mathrm{Im}f$ for all $t$ such that $t_{\mathrm{min}} \le t \le t_{\mathrm{max}}$, or equivalently $g([t_{\mathrm{min}}, t_{\mathrm{max}}]) \subset \mathrm{Im}f$. To show this, we notice $g(t_{\mathrm{min}}), g(t_{\mathrm{max}}) \in \mathrm{Im}f$. Because $f$ is a geodesic, we have
\begin{align*}
g(t_{\mathrm{min}}) &= f(d(a, g(t_{\mathrm{min}}))), &
g(t_{\mathrm{max}}) &= f(d(a, g(t_{\mathrm{max}}))).
\end{align*}
Replacing a part of $g$ with a part of $f$, we define a map $g' : [0, d(x, y)] \to X$ by
$$
g'(t)
=
\left\{
\begin{array}{ll}
g(t), & (0 \le t \le t_{\mathrm{min}}) \\
f(t - t_{\mathrm{min}} + d(a, g(t_{\mathrm{min}})), &
(t_{\mathrm{min}} \le t \le t_{\mathrm{max}}) \\
g(t). & (t_{\mathrm{max}} \le t \le d(x, y))
\end{array}
\right.
$$
This is well-defined, since $t_{\mathrm{max}} - t_{\mathrm{min}} = d(g(t_{\mathrm{max}}), g(t_{\mathrm{min}}))$ holds for the geodesic $g$, and the points $a$, $g(t_{\mathrm{min}})$ and $g(t_{\mathrm{max}})$ lie on the geodesic $f$. Furthermore, noting $a < x < y < b$ and applying the same argument as in the proof of Lemma \ref{lem:geodesic_from_x_y_z}, we can prove that $g'$ is a geodesic joining $x$ to $y$. Now, under Assumption \ref{assumption:for_intersection_number}, the geodesic $g'$ agrees with $g$. It follows that 
$$
g([t_{\mathrm{min}}, t_{\mathrm{max}}]) 
= g'([t_{\mathrm{min}}, t_{\mathrm{max}}]) 
= f([d(a, g(t_{\mathrm{min}})), d(a, g(t_{\mathrm{min}}))] 
\subset \mathrm{Im}f,
$$
and the present lemma is established.
\end{proof}

The following lemma shows that a magnitude $3$-chain of our interests can be ``subdivided'' up to boundaries.

\begin{lem} \label{lem:subdivision_of_3_chain}
Let $(X, d)$ be a metric space, and $a, b \in X$ two points. Let $\gamma = \langle a, x, y, b \rangle \in C^\ell_3(a, b)$ be a $3$-chain of length $\ell = d(a, b)$, and $g : [0, d(x, y)] \to X$ a geodesic joining $x$ to $y$. For any positive integer $n$, we choose $t_1, \cdots, t_n \in (0, d(x, y))$ such that $0 < t_1 < \cdots < t_n < d(x, y)$, and put $x_0 = x$, $x_{n+1} = y$, and $x_i = g(t_i)$ for $i = 1, \cdots, n$. Then $\langle a, x_i, x_{i+1}, x_n , b \rangle \in C^\ell_4(a, b)$ for $i = 0, \ldots, n$, and 
$$
\langle a, x_0, x_{n+1}, b \rangle
= \sum_{i = 0}^n \langle a, x_i, x_{i+1}, b \rangle
+  \sum_{i = 0}^{n-1} \partial \langle a, x_i, x_{i+1}, x_{n+1}, b \rangle.
$$
\end{lem}

\begin{proof}
The verification is straightforward.
\end{proof}

To make sense of the idea of ``intersection numbers'' mentioned at the begging of this subsection, we introduce a notion of regularity to $3$-chains.

\begin{dfn} \label{dfn:regular_3_chain}
Under Assumption \ref{assumption:for_intersection_number}, we say that a $3$-chain $\langle a, x, y, b \rangle \in C^\ell_3(a, b)$ is \textit{regular with respect to $f$} or \textit{$f$-regular} if the unique geodesic $g$ joining $x$ to $y$ satisfies either of the following:
\begin{enumerate}
\item[(i)]
The geodesic $g$ intersects with $f$ only at $y$. That is, $g([0, d(x, y))) \cap \mathrm{Im}f = \emptyset$ and $y = g(d(x, y)) \in \mathrm{Im}f$.

\item[(ii)]
The geodesic $g$ intersects with $f$ only at $x$. That is, $g((0, d(x, y)]) \cap \mathrm{Im}f = \emptyset$ and $x = g(0) \in \mathrm{Im}f$.

\item[(iii)]
The geodesic $g$ lies on $f$. That is $\mathrm{Im}g \subset \mathrm{Im}f$.

\item[(iv)]
The geodesic $g$ does not intersect $f$. That is $\mathrm{Im}g \cap \mathrm{Im}f = \emptyset$.

\end{enumerate}
\end{dfn}

\begin{lem} \label{lem:existence_of_reguar_chain}
Under Assumption \ref{assumption:for_intersection_number}, let $\gamma = \langle a, x, y, b \rangle \in C^\ell_3(a, b)$ be a $3$-chain. Then there exists an element $\beta \in C^\ell_4(a, b)$ such that 
$$
\gamma - \partial \beta
= \sum_{i = 0}^n \langle a, x_i, x_{i+1}, b \rangle
\in C^\ell_3(a, b)
$$ 
where $x_0 = x$ and  $x_{n+1} = y$; $x_1, \ldots, x_n$ are on the unique geodesic joining $x$ to $y$;  and $\langle a, x_i, x_{i+1}, b \rangle$ are $f$-regular for $i = 0, \cdots, n$.
\end{lem}

\begin{proof}
The idea of the proof is to ``subdivide'' $\gamma$ by Lemma \ref{lem:subdivision_of_3_chain} so that the resulting $3$-chains are $f$-regular. This idea is realized as follows: Let $g : [0, d(a, b)] \to X$ be the unique geodesic joining $x$ to $y$. If $\mathrm{Im}f \cap \mathrm{Im}g = \emptyset$, then $\gamma$ is already $f$-regular. Henceforth we assume $\mathrm{Im}f \cap \mathrm{Im}g \neq \emptyset$. Then we have $g^{-1}(\mathrm{Im}f) = [t_{\mathrm{min}}, t_{\mathrm{max}}] \subset [0, d(a, b)]$ by Lemma \ref{lem:geometric_intersection}. For a positive integer $n$, we can find $t_1, \ldots, t_n \in (0, d(x, y))$ such that $0 < t_1 < \cdots < t_n < d(x, y)$ and $\{ t_{\mathrm{min}}, t_{\mathrm{max}} \} \subset \{ t_0, t_1, \ldots, t_{n+1} \}$, where $t_0 = 0$ and $t_{n+1} = d(a, b)$. Setting $x_i = g(t_i)$ for $i = 0, 1, \ldots, n+1$, we get $f$-regular chains $\langle a, x_i, x_{i+1}, b \rangle$ for $i = 0, 1, \ldots, n$. Defining $\beta \in C^\ell_4(a, b)$ by $\beta = \sum_{i = 0}^{n-1}\langle a, x_i, x_{i+1}, x_{n+1}, b \rangle$ and applying Lemma \ref{lem:subdivision_of_3_chain}, we complete proof.
\end{proof}

\begin{dfn} \label{dfn:intersection_number}
Under Assumption \ref{assumption:for_intersection_number}, we make the following definitions. 
\begin{itemize}
\item[(a)]
For an $f$-regular $3$-chain $\langle a, x, y, b \rangle \in C^\ell_3(a, b)$, we put
$$
\nu_f(\langle a, x, y, b \rangle)
=
\left\{
\begin{array}{ll}
1, & (\mbox{the case of Definition \ref{dfn:regular_3_chain} (i)}) \\
0. & (\mbox{otherwise})
\end{array}
\right.
$$
If $\gamma = \sum_i k_i \gamma_i \in C^\ell_3(a, b)$ is a linear combination of $f$-regular chains $\gamma_i$, then we put $\nu_f(\gamma) = \sum_i k_i \nu_f(\gamma_i)$.

\item[(b)]
We define a homomorphism $\nu_f : C^\ell_3(a, b)/\partial(C^\ell_4(a, b)) \to \Z$ as follows: Let $\gamma = \langle a, x, y, b \rangle \in C^\ell_3(a, b)$ be a chain. Modulo a boundary, the $3$-chain $\gamma$ agrees with a linear combination 
$$
\eta = \sum_{i = 0}^n \langle a, x_i, x_{i+1}, b \rangle 
\in C^\ell_3(a, b)
$$
of $f$-regular $3$-chains $\langle a, x_i, x_{i+1}, b \rangle$ such that $x_1, \ldots, x_n$ are on the unique geodesic joining $x_0 = x$ to $x_{n+1} = y$, by Lemma \ref{lem:existence_of_reguar_chain}. We then put
$$
\nu_f([\gamma]) = \nu_f(\eta) 
= \sum_{i = 0}^n \nu_f(\langle a, x_i, x_{i+1}, b \rangle).
$$
Extending this definition linearly, we define the homomorphism $\nu_f$.
\end{itemize}
\end{dfn}

\begin{thm} 
Definition \ref{dfn:intersection_number} (b) is well-defined. 
\end{thm}

\begin{proof}
First, we need to check that, for $\gamma = \langle a, x, y, b \rangle \in C^\ell_3(a, b)$, the number $\nu_f([\gamma])$ is independent of the choice of a linear combination of $f$-regular $3$-chains
$$
\eta = \sum_{i = 0}^n\langle a, x_i, x_{i+1}, b \rangle
$$
which agrees with $\gamma$ modulo a boundary. Suppose that we have another choice
$$
\eta' = \sum_{i = 0}^{n'}\langle a, x'_i, x'_{i+1}, b \rangle.
$$
According to Definition \ref{dfn:intersection_number} (b), the points $x_i$ and $x'_i$ lie on the unique geodesic $g : [0, d(x, y)] \to X$ joining $x = x_0 = x'_0$ to $y = x_n = x'_{n'}$. If $\mathrm{Im}g \cap \mathrm{Im}f = \emptyset$, then we clearly have $\nu_f(\eta) = 0 = \nu_f(\eta')$. Henceforth we assume $\mathrm{Im}g \cap \mathrm{Im}f \neq \emptyset$, so that $g^{-1}(\mathrm{Im}f) = [t_{\mathrm{min}}, t_{\mathrm{max}}]$ by Lemma \ref{lem:geometric_intersection}. Since $\langle a, x_i, x_{i+1}, b \rangle$ and $\langle a, x'_i, x'_{i+1}, b \rangle$ are $f$-regular, we have $\{ g(t_{\mathrm{min}}), g(t_{\mathrm{max}}) \} \subset \{ x_0, \ldots, x_{n+1} \} \cap \{ x'_0, \ldots, x'_{n'+1} \}$. Now, if $g(t_{\mathrm{min}}) = x = x_0 = x'_0$, then $f$-regular chains in $\eta$ and $\eta'$ are of the types (ii), (iii) or (iv) in Definition \ref{dfn:regular_3_chain}, so that $\nu_f(\eta) = 0 = \nu_f(\eta')$. Otherwise, in each of $\eta$ and $\eta'$, there appears only one $f$-regular chain of type (i) in Definition \ref{dfn:regular_3_chain}, so that $\nu_f(\eta) = 1 = \nu_f(\eta')$. Consequently, $\nu_f([\gamma])$ is independent of the choice of a linear combination $\eta$ of $f$-regular chains as in Definition \ref{dfn:intersection_number} (b).

Next, we prove that $\nu_f([\gamma]) = 0$ if $\gamma = \partial \beta$ for some $\beta \in C^\ell_4(a, b)$. It suffices to consider the case that $\beta = \langle a, x, y, z, b \rangle$. Its boundary is
$$
\partial \beta = 
\langle a, x, z, b \rangle
- (\langle a, x, y, b \rangle + \langle a, y, z, b \rangle),
$$
where $x$, $y$ and $z$ lie on the unique geodesic joining $x$ to $z$. Now, we choose a linear combination $\eta$ of $f$-regular chains which agrees with $\langle a, x, y, b \rangle$ modulo a boundary as in Definition \ref{dfn:intersection_number} (b). We also choose a linear combination $\eta'$ of $f$-regular chains which agrees with $\langle a, y, z, b \rangle$ modulo a boundary as in Definition \ref{dfn:intersection_number} (b). Then $\eta + \eta'$ is a linear combination of $f$-regular chains which agrees with $\langle a, x, z, b \rangle$ modulo a boundary as in Definition \ref{dfn:intersection_number} (b). Therefore 
\begin{align*}
\nu_f(\langle a, x, z, b \rangle) &= \nu_f(\eta + \eta'), &
\nu_f(\langle a, x, y, b \rangle + \langle a, y, z, b \rangle)
= \nu_f(\eta) + \nu_f(\eta').
\end{align*}
Since $\nu_f$ is linear on $f$-regular chains by definition, $\nu_f(\partial \beta) = 0$. 
\end{proof}

Because $H^\ell_3(a, b) \subset C^\ell_3(a, b)/\partial(C^\ell_3(a, b))$, we get:

\begin{cor} \label{cor:intersection_number}
Under Assumption \ref{assumption:for_intersection_number}, there is a homomorphism
$$
\nu_f : H^\ell_3(a, b) \to \Z.
$$
\end{cor}

As an example, let $X$ consist of the vertices and the edges of the cube with edge length $r$, see Figure \ref{fig:cube}. For any $x, y \in X$, the distance $d(x, y)$ is defined by the length of the shortest path inside $X$ measured by the usual Euclidean distance on each edge of the cube. (Put differently, $(X, d)$ is the \textit{metric graph} \cite{B-H} associated to the box product $P_2^{\Box 3}$ of the two-point path graph $P_2$ whose edges are weighted by $r$.) We number the eight vertices. 
\begin{figure}[htb]
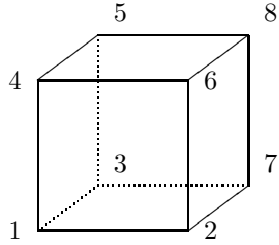

$$
\xygraph{
!{<0cm,0cm>;<1cm,0cm>:<0cm,1cm>::}
!{(0,0)}="a" ([]!{+(-0.3,0)} {1})
!{(2,0)}="b"  ([]!{+(0.3,0)} {2})
!{(0,2)}="c" ([]!{+(-0.3,0)} {4})
!{(2,2)}="d" ([]!{+(0.3,0)} {6})
!{(0.8,2.6)}="e" ([]!{+(0.3,0.3)} {5})
!{(2.8,2.6)}="f" ([]!{+(0.3,0.3)} {8})
!{(0.8,0.6)}="i" ([]!{+(0.3,0.3)} {3})
!{(2.8,0.6)}="j" ([]!{+(0.3,0.3)} {7})
"a"-"b"
"a"-"c"
"b"-"d"
"c"-"d"
"c"-"e"
"d"-"f"
"e"-"f"
"a"-@{.}"i"
"b"-"j"
"e"-@{.}"i"
"f"-"j"
"i"-@{.}"j"
}
$$
\caption{The vertices and the edges of a cube}
\label{fig:cube}
\end{figure}

With the distance, $(X, d)$ is a geodesic space. There are six geodesics joining $a = 1$ to $b = 8$, which are equivalent to each other in the equivalence relation in Definition \ref{dfn:equivalence_geodesic}. Thus, by Corollary \ref{cor:second_homology_on_geodesic_space}, we have $H^\ell_2(a, b) = 0$ for $\ell = d(a, b) = 3r$.

Under the choice $a = 1$, $b = 8$ and $\ell = d(a, b) = 3r$, we have the following magnitude $3$-cycle in $C^\ell_3(a, b)$
$$
\gamma 
= \langle a, 2, 7, b \rangle
- \langle a, 3, 7, b \rangle
+ \langle a, 3, 5, b \rangle
- \langle a, 4, 5, b \rangle
+ \langle a, 4, 6, b \rangle
- \langle a, 2, 6, b \rangle.
$$
To show that this cycle defines a non-trivial homology class in $H^\ell_3(a, b)$, we notice that Assumption \ref{assumption:for_intersection_number} is fulfilled in this case. As a reference geodesic $f$ joining $a$ to $b$, we choose that passes through $2$ and $7$. The cycle $\gamma$ is a linear combination of $f$-regular $3$-chains, and we readily see
$$
\nu_f(\gamma) = 0 - 1 + 0 - 0 + 0 - 0 = -1.
$$
This proves that the homology class $[\gamma] \in H^\ell_3(a, b)$ and also the homomorphism $\nu_f$ are non-trivial.


\section{Higher magnitude homology}
\label{sec:higher_magnitude_homology}

In this section, we prove Theorem \ref{thm:main} (Theorem \ref{thm:complete_description_under_non_branching}) that describes the magnitude homology of a geodesic space under Assumption \ref{assumption:non_branching}. The proof is a computation of the $E^2$-term of the smoothness spectral sequence, and the basic idea is to construct ``chainwise homotopy operators'' in the direct summands of the $E^1$-term. Note that the idea to construct a homotopy operator is ubiquitous, and is applied to other computations of magnitude homology (e.g.\ \cite{Gu}). A homotopy operator on a complex of chains may be constructed by an ``insertion of a point'' to chains. However, under restrictions about length and arrangements of singular points, we are often unable to choose an insertion point ``globally'' on the whole of a chain complex. Instead, we have to choose insertion points ``locally'' for each element of a chain complex to construct ``chainwise homotopy operators''. This makes our argument intricate. We begin with two extreme cases (\S\S\ref{subsec:E2_p1} and \S\S\ref{sec:cycle_associated_to_geodesics}), and then consider the general case (\S\S\ref{subsec:higher_magnitude_homology}).

\subsection{A vanishing result about $E^2_{p, 1}$}
\label{subsec:E2_p1}

We establish here a vanishing $E^2_{p, 1} = 0$ of the smoothness spectral sequence for the magnitude homology $H^\ell_*(X)$ under Assumption \ref{assumption:non_branching}. 

To begin with, we sketch the idea of the proof: In the present case, it seems impossible to construct homomorphisms $H : E^2_{p, 1} \to E^1_{p+1, 1}$ and $H : E^1_{p-1, 1} \to E^1_{p, 1}$ such that $d^1 H + H d^1 = - \mathrm{id}$. Instead, we construct $H\zeta \in E^1_{p+1, 1}$ and $Hd^1\zeta \in E^1_{p, 1}$ for each $\zeta \in E^2_{p, 1}$ such that $d^1H\zeta + Hd^1\zeta = - \zeta$ holds true. The construction of $H\zeta$ and $Hd^1\zeta$ goes roughly as follows: One can express $\zeta \in E^1_{p, 1}(a, b)$ with $a, b \in X$ as
\begin{align*}
\zeta &= \sum_{\gamma} N_\gamma \gamma, &
\gamma &= 
\langle a, x_1^\gamma, x_2^\gamma, x_3^\gamma, \cdots, x_p^\gamma, b \rangle,
\end{align*}
where $\gamma$ runs over a finite set of chains, and $N_\gamma \in \Z$ for each $\gamma$. Up to a boundary, we can assume that the first subchains $\langle a, x_1^\gamma, x_2^\gamma, x_3^\gamma \rangle$ of $\gamma$ are not $4$-cuts (Lemma \ref{lem:replacable_by_chains_starting_with_non_4_cut}). If we express $d^1\zeta = \sum_i (-1)^i d^1_i\zeta$, then the first subchains of $d^1_i\zeta$ are 
\begin{align*}
&\langle a, x_2^\gamma, x_3^\gamma \rangle, &
&\langle a, x_1^\gamma, x_2^\gamma \rangle, &
&\langle a, x_1^\gamma, x_2^\gamma \rangle. 
\end{align*}
We then introduce an equivalence relation to these first subchains (Definition \ref{dfn:triples_and_pairs}). Under Assumption \ref{assumption:non_branching}, a geodesic emanating from $a \in X$ is associated to each equivalence class. Then we can choose a point on each of these geodesics so that the insertion of these points to $\gamma$ and $d^1_i\gamma$ constructs $H\gamma$ and $H d^1_i\gamma$.

\medskip

Now, recall from \S\S\ref{subsec:spectral_sequence} that the $E^1$-term $(E^1_{p, q}, d^1)$ generally admits a direct sum decomposition
$$
E^1_{p, q} = \bigoplus_{\varphi \in \widehat{P}_q(X)} C^\ell_{p+q}(\varphi),
$$
where $\widehat{P}_q(X)$ is the set of (possibly improper) $q$-chains, and $C^\ell_{p+q}(\varphi)$ is the group whose generators are in the set $P^\ell_{p+q}(\varphi)$ of proper $(p+q)$-chains of length $\ell$ whose frames are $\varphi$. Applying this decomposition to the case that $q = 1$ and $\varphi = \langle a, b \rangle$ with any $a, b \in X$, we have
\begin{align*}
E^1_{p, 1}(a, b) &= C^\ell_{p+1}(\langle a, b \rangle), &
E^2_{p, 1}(a, b) &= H^\ell_{p+1}(\langle a, b \rangle),
\end{align*}
where $H^\ell_{p+1}(\langle a, b \rangle)$ is the homology of the subcomplex $C^\ell_{p+*}(\langle a, b \rangle)$ of $E^1_{p, *}$.

\begin{dfn}
Let $(X, d)$ be a metric space, $\ell$ a real number, $a, b \in X$ two points such that $\ell > d(a, b)$, and $E^1_{p, q}(a, b)$ the smoothness spectral sequence for the direct summand of the magnitude homology $H^\ell_*(a, b)$. 
\begin{itemize}
\item
For $p \ge 3$, we define ${}'P^\ell_{p+1}(\langle a, b \rangle) \subset P^\ell_{p+1}(\langle a, b \rangle)$ to be the subset consisting of $(p+1)$-chains $\langle a, x_1, \cdots, x_p, b \rangle$ such that the first subchain $\langle a, x_1, x_2, x_3 \rangle$ is not a $4$-cut. 

\item
We define ${}'E^1_{p, 1}(a, b) \subset E^1_{p, 1}(a, b)$ to be the subgroup generated by chains in ${}'P^\ell_{p+1}(\langle a, b \rangle)$ for $p \ge 3$, and the trivial group otherwise.
\end{itemize}
\end{dfn}

\begin{lem} \label{lem:replacable_by_chains_starting_with_non_4_cut}
Let $(X, d)$ be a geodesic space, $\ell$ a real number, $a, b \in X$ two points such that $\ell > d(a, b)$, and $E^1_{p, q}(a, b)$ the smoothness spectral sequence for the direct summand of the magnitude homology $H^\ell_*(a, b)$. Then any $(p+1)$-chain in $E^1_{p, 1}(a, b)$ is homologous to a linear combination of $(p+1)$-chains in ${}'E^1_{p, 1}(a, b)$.
\end{lem}

\begin{proof}
It is enough to consider the case of $p \ge 2$. Let $\gamma = \langle a, x_1, \cdots, x_p, b \rangle \in E^1_{p, 1}(a, b)$ be such that $\langle a, x_1, x_2, x_3 \rangle$ is a $4$-cut, where $x_3 = b$ when $p = 2$. Since $X$ is geodesic, there is a point $\overline{x}$ on the geodesic joining $a$ to $x_1$ such that $\langle \overline{x}, x_1, x_2, x_3 \rangle$ is a $4$-cut. We then have a $(p+2)$-chain $\beta = \langle a, \overline{x}, x_1, x_2, x_3, \cdots, x_p, b \rangle \in E^1_{p+1, 1}(a, b)$. Its boundary turns out to be
$$
d^1\beta = - \gamma + \sum_{i = 4}^p(-1)^i d^1_i \beta,
$$
where $d^1_i \beta \in P^\ell_{p+1}(\langle a, b \rangle)$ is given by
$$
d^1_i \beta =
\left\{
\begin{array}{ll}
\partial_i \beta, & (\sigma(\partial_i\beta) = p) \\
0. & (\sigma(\partial_i\beta) \neq p)
\end{array}
\right.
$$
For $i \ge 4$, if $\sigma(\partial_i\beta) = p$, then the first $3$-subchain of $d^1_i\beta = \partial_i\beta$ is $\langle a, \overline{x}, x_1, x_2 \rangle$, which is not a $4$-cut. Hence $\gamma + d^1\beta \in {}'E^1_{p, 1}(a, b)$. We remark that, if the subchains $\langle x_i, x_{i+1}, x_{i+2}, x_{i+3} \rangle$, ($i = 0, \ldots, p-2$, $x_0 = a$) are all $4$-cuts, then $d^1\beta = - \gamma$. 
\end{proof}

\begin{dfn} \label{dfn:triples_and_pairs}
Let $(X, d)$ be a metric space, $\ell$ a real number, $a, b \in X$ two points such that $\ell > d(a, b)$, and $E^1_{p, q}(a, b)$ the smoothness spectral sequence for the direct summand of the magnitude homology $H^\ell_*(a, b)$. For $p \ge 3$, we express a given element $\zeta \in {}'E^1_{p, 1}(a, b)$ as follows
\begin{align*}
\zeta &=
\sum_{\gamma \in \F(\zeta)} N_\gamma \gamma, 
&
\gamma &= 
\langle a, x^\gamma_1, x^\gamma_2, x^\gamma_3, \cdots, x^\gamma_p, b \rangle,
\end{align*}
where $\F(\zeta) \subset {}'P^\ell_{p+1}(\langle a, b \rangle)$ is a finite set, and $N_\gamma \in \Z$. 
\begin{enumerate}
\item[(a)]
We define $\F_3(\zeta)$ and $\F_2(\zeta)$ by
\begin{align*}
\F_3(\zeta)
&= 
\{ (x^\gamma_1, x^\gamma_2, x^\gamma_3) \in X^3 |\ \gamma \in \F(\zeta) \}, \\
\F_2(\zeta)
&=
\{ (x^\gamma_1, x^\gamma_2), (x^\gamma_2, x^\gamma_3), (x^\gamma_1, x^\gamma_3)
\in X^2 |\ \gamma \in \F(\zeta) \}.
\end{align*}

\item[(b)]
We define maps $\delta_1, \delta_2, \delta_3 : \F_3(\zeta) \to \F_2(\zeta)$ by
\begin{align*}
\delta_1((x_1, x_2, x_3))
&= (x_2, x_3), &
\delta_2((x_1, x_2, x_3))
&= (x_1, x_3), &
\delta_3((x_1, x_2, x_3))
&= (x_1, x_2).
\end{align*}

\item[(c)]
We generate an equivalence relation $\approx$ on $\F_3(\zeta) \sqcup \F_2(\zeta)$ by the following two relations $\sim$:
\begin{itemize}
\item
We have $(x_1, x_2, x_3) \sim (x'_1, x'_2, x'_3)$ for triples in $\F_3(\zeta)$ if 
$$
\delta_i((x_1, x_2, x_3)) 
= \delta_j((x'_1, x'_2, x'_3))
$$
for some pair $(i, j)$ of $i = 1, 2, 3$ and $j = 1, 2, 3$. In other words, $\{ x_1, x_2, x_3 \} \cap \{ x'_1, x'_2, x'_3 \}$ consists of two or three points.

\item
We have $(x_1, x_2) \sim (x'_1, x'_2, x'_3)$ for a pair in $\F_2(\zeta)$ and a triple in $\F_3(\zeta)$ if $(x_1, x_2) = \delta_i((x'_1, x'_2, x'_3))$ for $i = 1, 2$ or $3$. In other words, $\{ x_1, x_2 \} \cap \{ x'_1, x'_2, x'_3 \}$ consists of two points.

\end{itemize}

\end{enumerate}
\end{dfn}

\begin{lem} \label{lem:equivalence_rel_and_unique_geodesic}
Under Assumption \ref{assumption:non_branching}, we consider Definition \ref{dfn:triples_and_pairs}. Suppose that $(x_1, x_2, x_3), (x'_1, x'_2, x'_3) \in \F_3(\zeta)$ are triples such that $(x_1, x_2, x_3) \sim (x'_1, x'_2, x'_3)$ and $d(a, x_1) \le d(a, x_1')$. Then there exists a unique geodesic joining $a$ to $x'_1$ on which $x_1$ lies.
\end{lem}

\begin{proof}
We examine the possible nine cases arising from $(x_1, x_2, x_3) \sim (x'_1, x'_2, x'_3)$. 
\begin{enumerate}
\item[(1-1)]
In the case that $(x_2, x_3) = (x'_2, x'_3)$, since $\langle a, x_1, x_2, x_3 \rangle$ is not a $4$-cut, Lemma \ref{lem:geodesic_from_x_y_z} constructs a geodesic joining $a$ to $x_3 = x'_3$ on which $x_1$ and $x_2 = x'_2$ lie. We also have a geodesic joining $a$ to $x_3 = x'_3$ on which $x'_1$ and $x_2 = x'_2$ lie. But, these geodesics are the same by Assumption \ref{assumption:non_branching}. Thus, by restriction, we get a geodesic joining $a$ to $x'_1$ on which $x_1$ lies. This is unique by Assumption \ref{assumption:non_branching}.

\item[(1-2)]
In the case that $(x_2, x_3) = (x'_1, x'_3)$, we have a geodesic $g$ joining $a$ to $x_3$ on which $x_1$ and $x_2$ lie, and a geodesic $g'$ joining $a$ to $x'_3$ on which $x'_1$ lies. Since $x_2 = x'_1$ and $x_3 = x'_3$, Assumption \ref{assumption:non_branching} leads to $g = g'$. By restriction, we get a geodesic joining $a$ to $x_2 = x'_1$ on which $x_1$ lies. This is unique by Assumption \ref{assumption:non_branching}.

\item[(1-3)]
In the case that $(x_2, x_3) = (x'_1, x'_2)$, the argument in (1-2) applies.

\item[(2-1)]
The case that $(x_1, x_3) = (x'_2, x'_3)$ does not occur under the assumption $d(a, x_1) \le d(a, x'_1)$.

\item[(2-2)]
In the case that $(x_1, x_3) = (x'_1, x'_3)$, a geodesic joining $a$ to $x_1 = x'_1$ turns out to be unique under Assumption \ref{assumption:non_branching}: Let $g$ and $g'$ be geodesics joining $a$ to $x_1 = x'_1$. Since $a < x_1 < x_3$, we can prolong $g$ to get a geodesic $\tilde{g}$ joining $a$ to $x_3$ on which $x_1$ lies. Similarly, prolonging $g'$ we get a geodesic $\tilde{g}'$ joining $a$ to $x'_3$ on which $x'_1$ lies. Because $x_1 = x'_1$ and $x_3 = x'_3$, we have $\tilde{g} = \tilde{g}'$ by Assumption \ref{assumption:non_branching}, and hence $g = g'$.

\item[(2-3)]
In the case that $(x_1, x_3) = (x'_1, x'_2)$, the argument in (2-2) applies.

\item[(3-1)]
The case that $(x_1, x_2) = (x'_2, x'_3)$ does not occur under the assumption $d(a, x_1) \le d(a, x'_1)$.

\item[(3-2)]
In the case that $(x_1, x_2) = (x'_1, x'_3)$, the argument in (2-2) applies.

\item[(3-3)]
In the case that $(x_1, x_2) = (x'_1, x'_2)$, the argument in (2-2) applies.

\end{enumerate}
\end{proof}

\begin{lem} \label{lem:uniqueness_from_non_branching}
Let $(X, d)$ be a geodesic metric space satisfying Assumption \ref{assumption:non_branching}, and $f$ a geodesic joining a point $a \in X$ to another point $x \in X$. For any $t \in (0, d(a, x))$ and any geodesic $g$ joining $a$ to $f(t)$, we have $f|_{[0, t]} = g$.
\end{lem}

\begin{proof}
Since $f$ is a geodesic, we have $a < f(t) < x$. By restricting $f$, we get a geodesic $h$ joining $f(t)$ to $x$. Then the concatenation of $g$ and $h$ at $f(t)$ constructs a path $f'$ from $a$ to $x$. This pass $f'$ is a geodesic joining $a$ to $x$ on which $f(t)$ lies by the argument in the proof of Lemma \ref{lem:geodesic_from_x_y_z}. Now $f = f'$ follows from Assumption \ref{assumption:non_branching}, so that $f|_{[0, t]} = f'|_{[0, t]} = g$.
\end{proof}

\begin{lem} \label{lem:transitivity_of_property_wtr_relation}
Under Assumption \ref{assumption:non_branching}, we consider Definition \ref{dfn:triples_and_pairs}. Suppose that triples $(x_1, x_2, x_3), (x'_1, x'_2, x'_3), (x''_1, x''_2, x''_3) \in \F_3(\zeta)$ satisfy $d(a, x_1) \le d(a, x'_1) \le d(a, x''_1)$ and either of the following:
\begin{itemize}
\item[(i)]
$(x_1, x_2, x_3) \sim (x'_1, x'_2, x'_3) \sim (x''_1, x''_2, x''_3)$,

\item[(ii)]
$(x_1, x_2, x_3) \sim (x''_1, x''_2, x''_3) \sim (x'_1, x'_2, x'_3)$,

\item[(iii)]
$(x'_1, x'_2, x'_3) \sim (x_1, x_2, x_3) \sim (x''_1, x''_2, x''_3)$.

\end{itemize}
Then $x_1$ and $x'_1$ lie on a unique geodesic joining $a$ to $x''_1$. 
\end{lem}

\begin{proof}
In the case of (i), Lemma \ref{lem:equivalence_rel_and_unique_geodesic} provides us a unique geodesic $g$ joining $a$ to $x'_1$ on which $x_1$ lies. We also have a unique geodesic $g'$ joining $a$ to $x''_1$ on which $x'_1$ lies. By Lemma \ref{lem:uniqueness_from_non_branching}, we have $g'|_{[0, d(a, x_1')]} = g$. Hence $x_1$ lies on the unique $g'$. 

In the case of (ii), Lemma \ref{lem:equivalence_rel_and_unique_geodesic} provides us a unique geodesic $g$ joining $a$ to $x''_1$ on which $x_1$ lies, and also a unique geodesic $g'$ joining $a$ to $x''_1$ on which $x'_1$ lies. If $x_1 = x''_1$ or $x''_1 = x'_1$, then $g = g'$ follows from Assumption \ref{assumption:non_branching}. If $x_1 \neq x'_1 \neq x''_1$, then the geodesic $g'$ can be extended to a geodesic $\tilde{g}'$ joining $a$ to a point $z$ different from $x''_1$, according to the examination in the proof of Lemma \ref{lem:equivalence_rel_and_unique_geodesic} with respect to the first relation in (ii). Now Lemma \ref{lem:uniqueness_from_non_branching} implies $g' = \tilde{g}'|_{[0, d(a, x_1'')]} = g$, and both $x_1$ and $x'_1$ lie on this unique geodesic $g' = g$ joining $x$ to $x_1''$.

In the case of (iii), Lemma \ref{lem:equivalence_rel_and_unique_geodesic} provides us a unique geodesic $g$ joining $a$ to $x'_1$ on which $x_1$ lies, and also a unique geodesic $g'$ joining $a$ to $x''_1$ on which $x_1$ lies. If $x_1 = x'_1$, then $g$ is a part of $g''$ by Lemma \ref{lem:uniqueness_from_non_branching}. If $x'_1 = x''_1$, then $g = g'$ under Assumption \ref{assumption:non_branching}. If $x_1 \neq x_1' \neq x''_1$, then, in view of the proof of Lemma \ref{lem:equivalence_rel_and_unique_geodesic}, the first relation in (iii) implies that $(x_2, x_3) = (x'_2, x'_3), (x'_1, x'_3)$ or $(x'_1, x'_2)$. Similarly, the second relation in (iii) implies that $(x_2, x_3) = (x''_2, x''_3), (x''_1, x''_3)$ or $(x''_1, x''_2)$. Among the nine cases arising from the combination of these three cases, four cases do not occur. In the remaining five cases, we can find a point $z \neq x_1, x'_1, x''_1$ such that $g$ and $g'$ respectively prolong to geodesics $\tilde{g}$ and $\tilde{g}'$ joining $a$ to $z$. Since $x_1$ lies on both $\tilde{g}$ and $\tilde{g}'$, we have $\tilde{g} = \tilde{g}'$ under Assumption \ref{assumption:non_branching}. It follows that $x'_1$ lies on the geodesic $g'$, and $g$ is a part of $g'$. 
\end{proof}

\begin{lem} \label{lem:geodesic_associated_to_equivalence_class}
Under Assumption \ref{assumption:non_branching}, we consider Definition \ref{dfn:triples_and_pairs}. For each equivalence class $\Lambda \in (\F_3(\zeta) \sqcup \F_2(\zeta))/\!\!\approx$, there exists a unique geodesic $g_\Lambda$ emanating from $a \in X$ such that $x_1$ lies on $g_\Lambda$ for any triple $(x_1, x_2, x_3) \in \F_3(\zeta)$ in the equivalence class $\Lambda$.
\end{lem}

\begin{proof}
Let $\{ (x_1^i, x_2^i, x_3^i) \}_{i \in I} \subset \F_3(\zeta)$ be the set of triples in the equivalence class $\Lambda$. The index set $I$ is finite. Let $j \in I$ be such that $d(a, x_1^i) \le d(a, x_1^j)$ for all $i \in I$. Then $g_\Lambda$ is given by a geodesic joining $a$ to $x_1^j$: If $I$ consists of at most two elements, then $x_1^i$ lies on $g_\Lambda$ for all $i \in I$ by Lemma \ref{lem:equivalence_rel_and_unique_geodesic}. In general, for any $i \in I$, there is a sequence of triples in $\{ (x_1^i, x_2^i, x_3^i) \}_{i \in I}$ such that its consecutive triples are related by $\sim$. Using Lemma \ref{lem:uniqueness_from_non_branching} and Lemma \ref{lem:transitivity_of_property_wtr_relation}, we can see that $x_1^i$ lies on $g_\Lambda$.
\end{proof}

\begin{lem} \label{lem:consistent_choice_of_base_points}
Under Assumption \ref{assumption:non_branching}, we consider Definition \ref{dfn:triples_and_pairs}. Then there are $\overline{x}^{(x_1, x_2, x_3)} \in X$ for $(x_1, x_2, x_3) \in \F_3(\zeta)$ and $\overline{x}^{(x_1, x_2)} \in X$ for $(x_1, x_2) \in \F_2(\zeta)$ which have the following properties.
\begin{itemize}
\item[(a)]
$a < \overline{x}^{(x_1, x_2, x_3)} < x_1$ for $(x_1, x_2, x_3) \in \F_3(\zeta)$.

\item[(b)]
$a < \overline{x}^{(x_1, x_2)} < x_1$ for $(x_1, x_2) \in \F_2(\zeta)$

\item[(c)]
If $\lambda \approx \mu \in \F_3(\zeta) \sqcup \F_2(\zeta)$, then $\overline{x}^\lambda = \overline{x}^\mu$.

\end{itemize}
\end{lem}

\begin{proof}
We choose the points in question, focusing on the equivalence classes in $\F_3(\zeta) \sqcup \F_2(\zeta)$. Let $\Lambda \in (\F_3(\zeta) \sqcup \F_2(\zeta))/\!\!\approx$ be an equivalence class. Suppose that $\{ (x^i_1, x^i_2, x^i_3) \}_{i \in I} \subset \F_3(\zeta)$ consists of all the representatives of $\Lambda$. By Lemma \ref{lem:geodesic_associated_to_equivalence_class}, we have a geodesic $g_\Lambda$ which emanates from $a$ (and terminates at $x_1^j$ such that $d(a, x_1^i) \le d(a, x_1^j)$ for all $i \in I$), and $x_1^i$ lies on $g_\Lambda$ for all $i \in I$. Let $k \in I$ be such that $d(a, x_1^k) \le d(a, x_1^i)$ for all $i \in I$, and $\overline{x}^\Lambda$ a point on $g_\Lambda$ such that $a < \overline{x}^\Lambda < x_1^k$. If we put $\overline{x}^{(x^i_1, x^i_2, x^i_3)} = \overline{x}^{\Lambda}$, then (a) is satisfied. If $(y_1, y_2) \in \F_2(\zeta)$ is in $\Lambda$, then there is $z$ such that $(z, y_1, y_2)$, $(y_1, z, y_2)$ or $(y_1, y_2, z)$ belongs to $\{ (x^i_1, x^i_2, x^i_3) \}_{i \in I} \subset \F_3(\zeta)$. In any event, we have $a < \overline{x}^\Lambda < y_1$. If we put $\overline{x}^{(y_1, y_2)} = \overline{x}^\Lambda$, then (b) is satisfied. By design, (c) is satisfied for all triples $(x_1, x_2, x_3)$ and pairs $(y_1, y_2)$ whose equivalence classes are $\Lambda$. Applying this construction to the other equivalence classes, we complete the proof.
\end{proof}

\begin{lem} \label{lem:local_homotopy_for_E2_p1}
Under Assumption \ref{assumption:non_branching}, let $\ell$ be a real number, $a, b \in X$ two points such that $\ell > d(a, b)$, and $E^r_{p, q}(a, b)$ the smoothness spectral sequence for the direct summand $H^\ell_*(a, b)$ of the magnitude homology. Then, for any $p \ge 3$ and any $(p+1)$-chain $\zeta \in {}'E^1_{p, 1}(a, b)$, there are $H\zeta \in E^1_{p+1, 1}(a, b)$ and $Hd^1\zeta \in E^1_{p, 1}(a, b)$ such that $d^1H\zeta + Hd^1\zeta = - \zeta$. 
\end{lem}

\begin{proof}
Classifying the chains in $\F(\zeta)$ by $\F_3(\zeta)$, we can express $\zeta$ as follows
$$
\zeta = 
\sum_{\lambda}
\sum_{i = 1}^{n_\lambda} N^\lambda_i
\langle a, x^\lambda_1, x^\lambda_2, x^\lambda_3, 
x_4^{\lambda, i}, x_5^{\lambda, i}, \cdots,
x_p^{\lambda, i}, b \rangle,
$$
where $\lambda = (x^\lambda_1, x^\lambda_2, x^\lambda_3)$ runs over $\F_3(\zeta)$. By definition, $\langle a, x^\lambda_1, x^\lambda_2, x^\lambda_3 \rangle$ is not a $4$-cut for every $\lambda$. We define $H\zeta \in E^1_{p+1, 1}(a, b)$ to be
$$
H\zeta
=
\sum_{\lambda}
\sum_{i = 1}^{n_\lambda} N^\lambda_i
\langle a, \overline{x}^\lambda, x^\lambda_1, x^\lambda_2, x^\lambda_3, 
x_4^{\lambda, i}, x_5^{\lambda, i}, \cdots,
x_p^{\lambda, i}, b \rangle,
$$
by using the points $\overline{x}^\lambda$ in Lemma \ref{lem:consistent_choice_of_base_points}. Notice that neither $\langle a, \overline{x}^\lambda, x^\lambda_1, x^\lambda_2 \rangle$ nor $\langle \overline{x}^\lambda, x^\lambda_1, x^\lambda_2, x^\lambda_3 \rangle$ is a $4$-cut. For a $(p-1)$-chain $\langle a, x_1, x_2, \cdots, x_{p-1}, b \rangle \in E^1_{p-1, 1}(a, b)$ such that $\mu = (x_1, x_2) \in \F_2(\zeta)$, we put
$$
H\langle a, x_1, x_2, \cdots, x_{p-1}, b \rangle
= \langle a, \overline{x}^\mu, x_1, x_2, \cdots, x_{p-1}, b \rangle
\in E^1_{p, 1}(a, b).
$$
Extending this linearly, we define $Hd^1\zeta \in E^1_{p, 1}(a, b)$. By Lemma \ref{lem:consistent_choice_of_base_points}, it holds that $\overline{x}^{\delta_k(\lambda)} = \overline{x}^\lambda$ for $\lambda \in \F_3(\zeta)$ and $k = 1, 2, 3$. Thanks to this property, we can verify $d^1H\zeta + Hd^1\zeta = - \zeta$ by direct computations. 
\end{proof}

\begin{prop} \label{prop:E2_p1}
Under Assumption \ref{assumption:non_branching}, let $\ell$ be a real number, and $E^r_{p, q}(a, b)$ the smoothness spectral sequence for the direct summand $H^\ell_*(a, b)$ of the magnitude homology with $a, b \in X$. If $\ell > d(a, b)$, then $E^2_{p, 1}(a, b) = 0$ for all $p$.
\end{prop}

\begin{proof}
In the case of $p \le 1$, we have nothing to prove. The case of $p = 2$ is already proved in Proposition \ref{prop:vanishing_E2_21_ell_general}. Therefore we consider the case of $p \ge 3$. We prove in this case that any $(p+1)$-cycle $\zeta \in E^1_{p, 1}(a, b)$ is null-homologous. We can assume $\zeta \in {}'E^1_{p, 1}(a, b)$ by Lemma \ref{lem:replacable_by_chains_starting_with_non_4_cut}. Then Lemma \ref{lem:local_homotopy_for_E2_p1} leads to $\zeta = - d^1H\zeta$. 
\end{proof}

\subsection{A cycle associated to geodesics}
\label{sec:cycle_associated_to_geodesics}

The aim of this subsection is to construct a cycle in a direct summand of the chain complex $(E^1, d^1)$. These cycles are associated to a choice of geodesics, and will represent the non-trivial magnitude homology classes in the main theorem of this section (Theorem \ref{thm:complete_description_under_non_branching}, which is the same as Theorem \ref{thm:main} in \S\ref{sec:introduction}.)

\begin{dfn} \label{dfn:magnitude_cycle}
Let $(X, d)$ be a metric space, $q$ a positive integer, $\ell_1, \ldots, \ell_q$ positive real numbers, and $\varphi = \langle \varphi_0, \cdots, \varphi_q \rangle \in P_q(X)$ a proper $q$-chain. 
\begin{itemize}
\item[(a)]
An \textit{admissible set} $\{ x_1, x'_1; \ldots; x_q, x'_q \}$ for $\varphi$ consists of $x_1, x'_1, \ldots, x_q, x'_q \in X$ such that
\begin{align*}
&\varphi_{i-1} < x_i < \varphi_i, &
&x_i \not< \varphi_i \not< x_{i+1}, &
&x_i \not< \varphi_i \not< x'_{i+1}, \\
&\varphi_{i-1} < x'_i < \varphi_i, &
&x'_i \not< \varphi_i \not< x_{i+1}, &
&x'_i \not< \varphi_i \not< x'_{i+1},
\end{align*}
for $i = 1, \ldots, q-1$.

\item[(b)]
Given an admissible set $\{ x_1, x'_1; \ldots; x_q, x'_q \}$ for $\varphi$, we define
$$
\gamma(x_1, x'_1; \cdots; x_q, x'_q)
= \langle 
\varphi_0, x_1 - x'_1, \varphi_1, 
\cdots, \varphi_{q-1},
x_q - x'_q, \varphi_q
\rangle
\in C^{\ell_1, \ldots, \ell_q}_{2q}(\varphi),
$$
where $\ell_1 = d(\varphi_0, \varphi_1), \ldots, \ell_q = d(\varphi_{q-1}, \varphi_q)$.

\end{itemize}
\end{dfn}

In the above, the notation as noted in Remark \ref{rem:notation_in_noncommutative_polynomial} is applied. For instance,
\begin{align*}
\langle \varphi_0, x_1 - x'_1, \varphi_1 \rangle
&= \langle \varphi_0, x_1, \varphi_1 \rangle
- \langle \varphi_0, x'_1, \varphi_1 \rangle, \\
\langle \varphi_0, x_1 - x'_1, \varphi_1, x_2 - x'_2, \varphi_2 \rangle
&= \langle \varphi_0, x_1, \varphi_1, x_2, \varphi_2 \rangle
- \langle \varphi_0, x_1, \varphi_1, x'_2, \varphi_2 \rangle \\
&- \langle \varphi_0, x'_1, \varphi_1, x_2, \varphi_2 \rangle
+ \langle \varphi_0, x'_1, \varphi_1, x'_2, \varphi_2 \rangle.
\end{align*}

\begin{lem}
The $2q$-chain $\gamma(x_1, x'_1; \cdots; x_q, x'_q) \in C^{\ell_1, \ldots, \ell_q}_{2q}(\varphi)$ in Definition \ref{dfn:magnitude_cycle} satisfies the following
\begin{gather*}
d^1\gamma(x_1, x'_1; \cdots; x_q, x'_q) 
= \partial\gamma(x_1, x'_1; \cdots; x_q, x'_q) = 0, \\
\gamma(\cdots; x_{i-1}, x'_{i-1}; x_i, x'_i; x_{i+1}, x'_{i+1}; 
\cdots)
= - 
\gamma(\cdots; x_{i-1}, x'_{i-1}; x'_i, x_i; x_{i+1}, x'_{i+1}; 
\cdots).
\end{gather*}
\end{lem}

\begin{proof}
The verification is straightforward.
\end{proof}

\begin{lem} \label{lem:homologous_cyclces}
In the setup of Definition \ref{dfn:magnitude_cycle}, suppose that
\begin{align*}
&\{ x_1, x'_1; \ldots; x_i, x'_i; \ldots; x_q, x'_q \}, &
&\{ x_1, x'_1; \ldots; \overline{x}_i, x'_i; \ldots; x_q, x'_q \}
\end{align*}
are admissible for $\varphi$. If $\varphi_{i-1} < x_i < \overline{x}_i < \varphi_i$ or $\varphi_{i-1} < \overline{x}_i < x_i < \varphi_i$ holds true, then the associated cycles
$$
\gamma(x_1, x'_1; \cdots; x_i, x'_i; \cdots; x_q, x'_q), \
\gamma(x_1, x'_1; \cdots; \overline{x}_i, x'_i; \cdots; x_q, x'_q) 
\in C^{\ell_1, \ldots, \ell_q}_{2q}(\varphi)
$$
are homologous with respect to $d^1$. 
\end{lem}

\begin{proof}
If $\varphi_{i-1} < x_i < \overline{x}_i < \varphi_i$, then we have an element
$$
\beta = 
\langle \varphi_0, x_1 - x'_1, \varphi_1, \cdots,
\varphi_{i-1}, x_i, \overline{x}_i, \varphi_{i+1}, \cdots, 
\varphi_{q-1}, x_q - x'_q, \varphi_q \rangle
\in C^{\ell_1, \ldots, \ell_q}_{2q+1}(\varphi)
$$
whose boundary $d^1\beta$ yields the difference of the cycles. The same construction applies to the case that $\varphi_{i-1} < \overline{x}_i < x_i < \varphi_i$.
\end{proof}

\begin{lem} \label{lem:existence_of_admissible_set}
In the setup of Definition \ref{dfn:magnitude_cycle}, we suppose that $(X, d)$ is a geodesic space satisfying Assumption \ref{assumption:non_branching}. Let $f_i^0, f_i^1 \in \Geod(\varphi_{i-1}, \varphi_i)$ be geodesics such that $f_i^0 \neq f_i^1$ for $i = 1, \ldots, q$. Then there are points $x^{f_i^0} \in f_i((0, \ell_i)), x^{f_i^1} \in f_i^1((0, \ell_i))$ such that $\{ x^{f_1^0}, x^{f_1^1}; \ldots; x^{f_q^0}, x^{f_q^1} \}$ is admissible for $\varphi$.
\end{lem}

\begin{proof}
We find the points in question by an induction: In the case of $q = 1$, we can arbitrarily choose $x^{f^j_1} \in f^j_1((0, \ell_1))$ for $j = 0, 1$. In the case of $q = 2$, we notice that $\varphi_0 \not< \varphi_1 \not< \varphi_2$. To see this relation, suppose $\varphi_0 < \varphi_1 < \varphi_2$. Then the concatenation of the geodesics $f_1^0$ and $f_2^0$ gives rise to a geodesic joining $\varphi_0$ to $\varphi_2$. Another geodesic joining $\varphi_0$ to $\varphi_2$ is given by the concatenation of $f_1^1$ and $f_2^0$. Since these geodesics share $\varphi_1$, Assumption \ref{assumption:non_branching} implies that they must be the same, and hence $f_1^0 = f_1^1$. This contradicts to $f_1^0 \neq f_1^1$, so that $\varphi_0 \not< \varphi_1 \not< \varphi_2$. Because of this relation, there are $x^{f^j_i} \in f^j_i((0, \ell_1))$ for $i = 1, 2$ and $j = 0, 1$ such that $x^{f^j_1} \not< \varphi_1 \not< x^{f^{j'}_2}$ for $j, j' \in \{ 0, 1 \}$. In the case of $q \ge 3$, let $k$ be such that $2 \le k \le q-1$. Suppose that we have $x^{f^j_i} \in f^j_i((0, \ell_i))$ such that $x^{f^j_i} \not< \varphi_i \not< x^{f^{j'}_{i+1}}$ for $i = 1, \ldots, k-1$ and $j, j' \in \{ 0, 1 \}$. Then it holds that $x^{f^j_k} \not< \varphi_k \not< \varphi_{k+1}$ for $j = 0, 1$. Actually, if we assume $x^{f^j_k} < \varphi_k < \varphi_{k+1}$, then we have a geodesic which joins $x^{f^j_k}$ to $\varphi_{k+1}$ and passes through $\varphi_k$ and a point on $f^0_{k+1}$ other than $\varphi_k$ or $\varphi_{k+1}$. We also have a geodesic which joins $x^{f^j_k}$ to $\varphi_{k+1}$ and passes through $\varphi_k$ and a point on $f^1_{k+1}$ other than $\varphi_k$ or $\varphi_{k+1}$. Under Assumption \ref{assumption:non_branching}, these geodesics must coincide, which contradicts to $f^0_{k+1} \neq f^1_{k+1}$. Thus $x^{f^j_k} \not< \varphi_k \not< \varphi_{k+1}$ for $j = 0, 1$. We can then find $x^{f^{j'}_{k+1}} \in f^{j'}_{k+1}((0, \ell_{k+1}))$ so that $x^{f^j_k} \not< \varphi_k \not< x^{f^{j'}_{k+1}}$ for $j, j' \in \{ 0, 1 \}$. Now, the induction works, and we end up with an admissible set for $\varphi$.
\end{proof}

We remark that in the setup of Lemma \ref{lem:existence_of_admissible_set}, all the points $\varphi_i$ are singular in $\varphi$.

\begin{prop}
In the setup of Definition \ref{dfn:magnitude_cycle}, we suppose that $(X, d)$ is a geodesic space satisfying Assumption \ref{assumption:non_branching}. Let $f^0_i, f^1_i \in \Geod(\varphi_{i-1}, \varphi_i)$ be geodesics such that $f_i^0 \neq f_i^1$ for $i = 1, \ldots, q$.
\begin{enumerate}
\item[(a)]
We have a well-defined homology class
$$
\Gamma(f^0_1, f^1_1; \cdots; f^0_q, f^1_q) 
\in H^{\ell_1, \ldots, \ell_q}_{2q}(\varphi)
$$
represented by the cycle $\gamma(x^{f_1^0}, x^{f_1^0}; \cdots; x^{f_q^0}, x^{f_q^1})$ associated to an admissible set $\{ x^{f_1^0}, x^{f_1^1}; \ldots, x^{f_q^0}, x^{f_q^1} \}$ for $\varphi$ such that $x^{f_i^j} \in f_i^j((0, \ell_i))$ for $i = 1, \ldots, q$ and $j = 0, 1$.

\item[(b)]
For $i = 1, \ldots, q$, it holds that
$$
\Gamma(\cdots; f_i^0, f_i^1; \cdots)
= - \Gamma(\cdots; f_i^1, f_i^0; \cdots).
$$
Further, if $f_k^2 \in \Geod(\varphi_{k-1}, \varphi_k)$ is such that $f_k^2 \neq f_k^0$ and $f_k^2 \neq f_k^1$, then
$$
\Gamma(\cdots; f_k^0, f_k^1; \cdots) + \Gamma(\cdots; f_k^1, f_k^2; \cdots)
= \Gamma(\cdots; f_k^0, f_k^2; \cdots).
$$

\end{enumerate}
\end{prop}

\begin{proof}
(a)
We write $\mathcal{X} = \{ x^{f^0_1}, x^{f^1_1}; \ldots, x^{f^0_q}, x^{f^1_q} \}$ for the given admissible set, and $\gamma(\mathcal{X})$ for the associated cycle. Let $\mathcal{Y} = \{ y^{f^0_1}, y^{f^1_1}; \ldots; y^{f^0_q}, y^{f^1_q} \}$ be another admissible set for $\varphi$ such that $y^{f^j_i} \in f^j_i((0, \ell_i))$ for $i = 1, \ldots, q$ and $j = 0, 1$. We will choose yet another admissible set $\mathcal{Z} = \{ z^{f^0_1}, z^{f^1_1}; \ldots; z^{f^0_q}, z^{f^1_q} \}$ for $\varphi$ such that $z^{f^j_i} \in f^j_i((0, \ell_i))$ for $i = 1, \ldots, q$ and $j = 0, 1$, and then prove that the cycle $\gamma(\mathcal{Z})$ is homologous to both $\gamma(\mathcal{X})$ and $\gamma(\mathcal{Y})$. 

We choose $\mathcal{Z}$ as follows: First, because $x^{f^j_q}$ and $y^{f^j_q}$ lie on the same geodesic $f^j_q$, we just choose $z^{f^j_q}$ so as to be $\varphi_{q-1} < z^{f^j_q} < x^{f^j_q}$ and $\varphi_{q-1} < z^{f^j_q} < y^{f^j_q}$ for $j = 0, 1$. Next, notice that Assumption \ref{assumption:non_branching} implies $\varphi_{q-2} \not< \varphi_{q-1} \not< z^{f^j_{q-1}}$ for $j = 0, 1$. We can then choose $z^{f^j_{q-1}}$ for $j = 0, 1$ such that
\begin{align*}
&\varphi_{q-2} < z^{f^j_{q-1}} < x^{f^j_{q-1}}, &
&\varphi_{q-2} < z^{f^j_{q-1}} < y^{f^j_{q-1}}
&z^{f^j_{q-1}} \not< \varphi_{q-1} \not< z^{f^{j'}_q}, &
\end{align*}
for $j, j' \in \{ 0, 1 \}$. We can iterate this choice, and eventually get an admissible set $\mathcal{Z}$ for $\varphi$ such that $\varphi_{i-1} < z^{f^j_i} < x^{f^j_i}$ and $\varphi_{i-1} < z^{f^j_i} < y^{f^j_i}$ for all $i = 1, \ldots, q$ and $j = 0, 1$. To prove that $\gamma(\mathcal{Z})$ and $\gamma(\mathcal{X})$ are homologous, observe that we have the following admissible sets for $\varphi$
\begin{align*}
\mathcal{W}_i
&=
\{ z^{f^0_1}, z^{f^1_1}; \ldots; z^{f^0_{i-1}}, z^{f^1_{i-1}};
z^{f^0_i}, x^{f^1_i}; x^{f^0_{i+1}}, x^{f^1_{i+1}}; \ldots;
x^{f^0_q}; x^{f^1_q} \}, \\
\mathcal{W}'_i
&=
\{ z^{f^0_1}, z^{f^1_1}; \ldots; z^{f^0_{i-1}}, z^{f^1_{i-1}};
z^{f^0_i}, z^{f^1_i}; x^{f^0_{i+1}}, x^{f^1_{i+1}}; \ldots;
x^{f^0_q}; x^{f^1_q} \}.
\end{align*}
By Lemma \ref{lem:homologous_cyclces}, the cycle $\gamma(\mathcal{W}_i)$ is homologous to $\gamma(\mathcal{W}'_i)$, and $\gamma(\mathcal{W}'_i)$ is homologous to $\gamma(\mathcal{W}_{i+1})$. Therefore $\gamma(\mathcal{X}) = \gamma(\mathcal{W}'_0)$ and $\gamma(\mathcal{Z}) = \gamma(\mathcal{W}'_q)$ are homologous. In the same way, we see that $\gamma(\mathcal{Y})$ is homologous to $\gamma(\mathcal{Z})$. 

(b)
The first formula is verified on the level of cycles. To show the second formula, we choose $x^{f^j_i} \in f^j_i((0, \ell_i))$ for $i = 1, \ldots, q$ and $j = 0, 1$ and $x^{f^2_k} \in f^2_i((0, \ell_k)$ so that
\begin{align*}
&\{ x^{f^0_1}, x^{f^1_1}; \cdots; x^{f^0_k}; x^{f^1_k}; \cdots;
x^{f^0_q}, x^{f^1_q}; \}, \\
&\{ x^{f^0_1}, x^{f^1_1}; \cdots; x^{f^1_k}; x^{f^2_k}; \cdots;
x^{f^0_q}, x^{f^1_q}; \}, \\
&\{ x^{f^0_1}, x^{f^1_1}; \cdots; x^{f^0_k}; x^{f^2_k}; \cdots;
x^{f^0_q}, x^{f^1_q}; \}
\end{align*}
are admissible sets for $\varphi$. This is possible by a generalization of the argument in the proof of Lemma \ref{lem:existence_of_admissible_set}. Then the formula is again verified on the level of cycles.
\end{proof}

\begin{rem}
Let $(X, d)$ be a geodesic space, and $\varphi = \langle \varphi_0, \varphi_1, \varphi_2 \rangle \in P^\ell_2(X)$ a proper $2$-chain of length $\ell > d(\varphi_0, \varphi_2)$. In this case, for any geodesics $f_1, f'_1 \in \Geod(\varphi_0, \varphi_1)$ and $f_2, f'_2 \in \Geod(\varphi_1, \varphi_2)$, we can always define $\Gamma(f_1, f'_1; f_2, f'_2) \in H^\ell_4(\varphi)$, without Assumption \ref{assumption:non_branching}. For this homology class, if $f_1$ is equivalent to $f'_1$, then it can happen that $\Gamma(f_1, f'_1; g_1, g'_2) = 0$ even if $f_1 \neq f'_1$ and $g_1 \neq g'_2$.
\end{rem}

\subsection{A description of higher magnitude homology}
\label{subsec:higher_magnitude_homology}

This subsection contains the proof of Theorem \ref{thm:main} (Theorem \ref{thm:complete_description_under_non_branching}). The proof will show that the only (possibly) non-trivial $E_2$-term of the smoothness spectral sequence for the magnitude homology $H^\ell_n(X)$ is 
$$
E^2_{q, q} 
= \bigoplus_{\varphi \in P_q(X)}
\bigoplus_{\stackrel{\ell_1, \ldots, \ell_q > 0}
{\ell_1 + \cdots + \ell_q = \ell}}
H^{\ell_1, \ldots, \ell_q}_{2q}(\varphi),
$$
and we identify the direct summand $H^{\ell_1, \ldots, \ell_q}_{2q}(\varphi)$ for $\varphi = \langle \varphi_0, \ldots, \varphi_q \rangle \in P_q(X)$ with
$$
H^{\ell_1, \ldots, \ell_q}_{2q}(\varphi)
=
\bigoplus_{f_i \neq \overline{f}_i}
\Z \Gamma(f_1, \overline{f}_1; \cdots; f_q; \overline{f}_q),
$$
where $f_i \in \Geod(\varphi_{i-1}, \varphi_i)$ runs over geodesics different from a reference geodesic $\overline{f}_i \in \Geod(\varphi_{i-1}, \varphi_i)$ for each $i = 1, \ldots, q$, and $\Gamma(f_1, \overline{f}_1; \cdots; f_q; \overline{f}_q)$ is represented by the following $2q$-chain introduced in \S\S\ref{sec:cycle_associated_to_geodesics}
$$
\gamma(x^{f_1}, x^{\overline{f}_1}; \cdots; x^{f_q}, x^{\overline{f}_q})
=
\langle 
\varphi_0, x^{f_1} - x^{\overline{f}_1},
\varphi_1, \cdots, 
\varphi_{q-1}, x^{f_q} - x^{\overline{f}_q}, \varphi_q
\rangle.
$$
Once this is shown (Proposition \ref{prop:determine_all_direct_summand}), we immediately get Theorem \ref{thm:main} (Theorem \ref{thm:complete_description_under_non_branching}). The proof of the above claim about the $E_2$-terms is to construct ``chainwise homotopy operators'' which relate a given element $\zeta \in C^{\ell_1, \ldots, \ell_q}_{p+q}(\varphi)$ with a linear combination of $\gamma(x^{f_1}, x^{\overline{f}_1}; \cdots; x^{f_q}, x^{\overline{f}_q})$. For this aim, we regard a linear combination of $\gamma(x^{f_1}, x^{\overline{f}_1}; \cdots; x^{f_q}, x^{\overline{f}_q})$ as ``standard'', and a linear combination of chains of the form
$$
\langle 
\varphi_0, x^{f_1} - x^{\overline{f}_1},
\varphi_1, \cdots, 
\varphi_{k-1}, x^{f_k} - x^{\overline{f}_k}, 
\varphi_k, x_1^{k+1}, \ldots, x^{k+1}_{m_{k+1}}, 
\varphi_{k+1}, \cdots,
\varphi_q
\rangle
$$
as an intermediate step which we call ``$k$-standard'' (Definition \ref{dfn:k_standard}). Then the ``chainwise homotopy operators'' are constructed by an induction: The initial step  (Lemma \ref{lem:induction_initial_step}) constructs ``chainwise homotopy operators'' which relate a given element $\zeta \in C^{\ell_1, \ldots, \ell_q}_{p+q}(\varphi)$ with a $1$-standard chain in the case that $\ell_1 = d(\varphi_0, \varphi_1)$. In the case that $\ell_1 > d(\varphi_0, \varphi_1)$, a null homotopy operator is constructed by applying the argument in \S\S\ref{subsec:E2_p1}. Then the inductive step (Lemma \ref{lem:induction_main_step}) constructs ``chainwise homotopy operators'' which relate a $k$-standard chain with a $(k+1)$-standard chain. The construction in this step is essentially the same with that in the initial step, up to a technical issue. To resolve this issue, we prove a lemma (Lemma \ref{lem:replacement_of_reference_points}) which is a replacement of points $x^{f}$ on geodesics $f$.

\bigskip

\begin{dfn} \label{dfn:k_standard}
Let $(X, d)$ be a geodesic space satisfying Assumption \ref{assumption:non_branching}, $q \ge 2$ an integer, $\ell_1,\ldots, \ell_q$ positive real numbers, and $\varphi = \langle \varphi_0, \cdots, \varphi_q \rangle \in \widehat{P}_q(X)$ a $q$-chain. Let $k$ be an integer such that $1 \le k \le q$. Suppose $d(\varphi_{i-1}, \varphi_i) = \ell_i$ for $i = 1, \ldots, k$. Suppose also that geodesics $\overline{f}_i \in \Geod(\varphi_{i-1}, \varphi_i)$ are given for $i = 1, \ldots, k$. We say that an element $\zeta \in C^{\ell_1, \ldots, \ell_q}_{p+q}(\varphi)$ is \textit{$k$-standard} if it has the following expression
\begin{multline*}
\zeta 
= \sum_{f_1, \ldots, f_k}
\sum_{\alpha}
N^{f_1, \ldots, f_k}_\alpha
\langle 
\varphi_0, x^{f_1} - x^{\overline{f}_1}, \varphi_1, 
\cdots \\
\cdot\cdot, 
\varphi_{k-1}, x^{f_k} - x^{\overline{f}_k}, \varphi_k,
x^{\alpha, k+1}_1, \ldots, x^{\alpha, k+1}_{m^\alpha_{k+1}}, \varphi_{k+1},
\cdots,
\varphi_{q-1}, x^{\alpha, q}_1, \ldots, x^{\alpha, q}_{m^\alpha_q}, \varphi_q
\rangle.
\end{multline*}
In the sum above,
\begin{itemize}
\item
$f_i$ runs over a finite set in $\Geod(\varphi_{i-1}, \varphi_i)$ such that $f_i \neq \overline{f}_i$ for $i = 1, \ldots, k$,

\item
$\alpha$ runs over a finite set.
\end{itemize}
The points $x^{f_i}$ and $x^{\overline{f}_i}$ lie on $f_i$ and $\overline{f}_i$, respectively, so that they form an admissible set $\{ x^{f_1}, x^{\overline{f}_1}; \ldots; x^{f_k}, x^{\overline{f}_k} \}$ for the subchain $\langle \varphi_0, \cdots, \varphi_k \rangle$. The integers $m^\alpha_{k+1}, \ldots, m^\alpha_q$ are non-negative and satisfy $k + m^\alpha_{k+1} + \cdots + m^\alpha_q = p$. That $m^\alpha_j = 0$ means there is no point between $\varphi_{j-1}$ and $\varphi_j$. 
\end{dfn}

When $p > q$, we regard a $q$-standard element $\zeta \in C^{\ell_1, \ldots, \ell_q}_{p+q}(\varphi)$ as trivial.

\begin{lem} \label{lem:induction_initial_step}
Let $(X, d)$ be a geodesic space satisfying Assumption \ref{assumption:non_branching}, $q \ge 2$ an integer, $\ell_1,\ldots, \ell_q$ positive real numbers, and $\varphi = \langle \varphi_0, \cdots, \varphi_q \rangle \in \widehat{P}_q(X)$ a $q$-chain. Suppose that we are given an element $\zeta \in C^{\ell_1, \ldots, \ell_q}_{p+q}(\varphi)$.
\begin{itemize}
\item[(a)]
If $\ell_1 > d(\varphi_0, \varphi_1)$, then there are $H\zeta \in C^{\ell_1, \ldots, \ell_q}_{p+q+1}(\varphi)$ and $Hd^1\zeta \in C^{\ell_1, \ldots, \ell_q}_{p+q}(\varphi)$ such that $d^1H\zeta + Hd^1\zeta = - \zeta$.

\item[(b)]
If $\ell_1 = d(\varphi_0, \varphi_1)$, then there are $H\zeta \in C^{\ell_1, \ldots, \ell_q}_{p+q+1}(\varphi)$ and $Hd^1\zeta \in C^{\ell_1, \ldots, \ell_q}_{p+q}(\varphi)$ such that $d^1H\zeta + Hd^1\zeta = \pi\zeta - \zeta$, where $\pi\zeta \in C^{\ell_1, \ldots, \ell_q}_{p+q}(\varphi)$ is $1$-standard. In particular, $\pi\zeta = 0$ if $p = 0$.

\end{itemize}
\end{lem}

\begin{proof}
(a)
In this case, let us express $\zeta$ as
$$
\zeta = \sum_\alpha N_\alpha
\langle \varphi_0, x^{\alpha, 1}_1, \ldots, x^{\alpha, 1}_{m^\alpha_1}, 
\varphi_1, \cdots \rangle.
$$
Note that $m^\alpha_1 \ge 2$ for all $\alpha$. Applying the argument in the proof of Lemma \ref{lem:replacable_by_chains_starting_with_non_4_cut} and Lemma \ref{lem:local_homotopy_for_E2_p1} to the subchains $\langle \varphi_0, x^{\alpha, 1}_1, \ldots, x^{\alpha, 1}_{m^\alpha_1}, \varphi_1 \rangle \in C^{\ell_1}_{m^\alpha_1 + 1}(\langle \varphi_0, \varphi_1 \rangle)$, we can construct $H\zeta \in C^{\ell_1, \ldots, \ell_q}_{p+q+1}(\varphi)$ and $Hd^1\zeta \in C^{\ell_1, \ldots, \ell_q}_{p+q}(\varphi)$ such that $d^1H\zeta + Hd^1\zeta = - \zeta$.

(b)
We can express $\zeta$ as $\zeta = \zeta_0 + \zeta_1 + \cdots$, where
\begin{align*}
\zeta_0 &=
\sum_{\alpha_0} N_{\alpha_0}
\langle \varphi_0, \varphi_1, 
x^{\alpha_0, 2}_1, \ldots, x^{\alpha_0, 2}_{m^{\alpha_0}_2}, \varphi_2, 
\cdots \rangle, \\
\zeta_j &= 
\sum_{\alpha_j} N_{\alpha_j}
\langle \varphi_0, x^{\alpha_j, 1}_1, \ldots, x^{\alpha_j, 1}_j, \varphi_1, 
x^{\alpha_j, 2}_1, \ldots, x^{\alpha_j, 2}_{m^{\alpha_j}_2}, \varphi_2, 
\cdots \rangle,
\end{align*}
where $\alpha_j$ runs over a finite set $\A_j$ for each $j$. If $j \ge 1$, then the points $x_1^{\alpha_1}, \ldots, x_j^{\alpha_j}$ lie on a geodesic joining $\varphi_0$ to $\varphi_1$, which is unique under Assumption \ref{assumption:non_branching}. Classifying the chains according to the geodesics joining $\varphi_0$ to $\varphi_1$, we can express $\zeta_j$ as
$$
\zeta_j =
\sum_f \sum_{\beta_j}
N^f_{\beta_j}
\langle 
\varphi_0, x^{f, \beta_j, 1}_1, \ldots, x^{f, \beta_j, 1}_j, \varphi_1, 
x^{f, \beta_j, 2}_1, \ldots, x^{f, \beta_j, 2}_{m^{\beta_j}_2}, 
\varphi_2, \cdots
\rangle,
$$
where $f$ runs over a finite set $\F_j \subset \Geod(\varphi_0, \varphi_1)$, $\beta_j$ over a finite set $\B_j$, and the points $x^{f, \beta_j, 1}_1, \ldots, x^{f, \beta_j, 1}_j$ lie on $f$. We put $\F = \{ \overline{f} \} \cup \bigcup_j \F_j$, which is also a finite subset in $\Geod(\varphi_0, \varphi_1)$. Here $\overline{f} \in \Geod(\varphi_0, \varphi_1)$ is a geodesic chosen as a reference, which may already be included in some $\F_j$. For each $f \in \F$, we can choose $\overline{x}^f \in f((0, \ell_1))$ such that $\varphi_0 < \overline{x}^f < x^{f, \beta, 1}_1$ for all $\beta \in \bigcup_{j \ge 1}\B_j$ and $\overline{x}^f \not< \varphi_1 \not< x^{\alpha_0, 2}_1$ for all $\alpha_0 \in \A_0$, where $x^{\alpha_0, 2}_1 = \varphi_2$ if $m^{\alpha_0}_2 = 0$. From $\zeta \in C^{\ell_1, \ldots, \ell_q}_{p+q}(\varphi)$, it follows that $\overline{x}^f \not< \varphi_1 \not< x^{f, \beta, 2}_1$ for all $f \in \F$ and $\beta \in \bigcup_{j \ge 1}\B_j$. Now, we put
\begin{align*}
H\zeta_0
&=
\sum_{\alpha_0} N_{\alpha_0}
\langle \varphi_0, \overline{x}^f, \varphi_1, 
x^{\alpha_0, 2}_1, \ldots, x^{\alpha_0, 2}_{m^{\alpha_0}_2}, \varphi_2, 
\cdots \rangle, \\
H\zeta_j &=
\sum_f \sum_{\beta_j}
N^f_{\beta_j}
\langle 
\varphi_0, \overline{x}^f, 
x^{f, \beta_j, 1}_1, \ldots, x^{f, \beta_j, 1}_j, \varphi_1, 
x^{f, \beta_j, 2}_1, \ldots, x^{f, \beta_j, 2}_{m^{\beta_j}_2}, 
\varphi_2, \cdots
\rangle,
\end{align*}
for $j \ge 1$. Then $H\zeta = H\zeta_0 + H\zeta_1 + \cdots \in C^{\ell_1, \ldots, \ell_q}_{p+q+1}(\varphi)$. The boundary $d^1\zeta$ is a linear combination of the chains of the following forms
\begin{align*}
&\langle \varphi_0, \varphi_1, x^2_1, \cdots \rangle, &
&\langle \varphi_0, x^f_1, \ldots, x^f_j, \varphi_1, 
x^2_1, \cdots \rangle,
\end{align*}
where $x^f_1, \ldots, x^f_j$ are on a geodesic $f \in \F$. For these chains appearing in $d^1\zeta$, we can make sense of the following chains, thanks to the nature of $\overline{x}^f$, 
\begin{align*}
H\langle \varphi_0, \varphi_1, x^2_1, \cdots, \rangle
&=
\langle \varphi_0, \overline{x}^{\overline{f}}, \varphi_1, x^2_1, 
\cdots \rangle, \\
H\langle \varphi_0, x^f_1, \ldots, x^f_j, \varphi_1, 
x^2_1, \cdots \rangle
&=
\langle \varphi_0, \overline{x}^f, x^f_1, \ldots, x^f_j, \varphi_1, 
x^2_1, \cdots \rangle.
\end{align*}
As a linear extension, we get $Hd^1\zeta = Hd^1\zeta_0 + Hd^1\zeta_1 + \cdots \in C^{\ell_1, \ldots, \ell_q}_{p+q}(\varphi)$. By direct computations, we find $d^1H\zeta_j + Hd^1\zeta_j = - \zeta_j$ for $j \neq 1$. We also find $d^1H\zeta_1 + Hd^1\zeta_1 = \pi\zeta_1 - \zeta_1$, where
$$
\pi\zeta_1
=
\sum_{f \neq \overline{f}} \sum_{\beta_1}
N^f_{\beta_1}
\langle 
\varphi_0, \overline{x}^f - \overline{x}^{\overline{f}}, \varphi_1, 
x^{f, \beta_1, 2}_1, \ldots, x^{f, \beta_1, 2}_{m^{\beta_1}_2}, 
\varphi_2, \cdots
\rangle.
$$
Clearly, $\pi\zeta = \pi\zeta_1$ is $1$-standard. If $p = 0$, then $\zeta = \zeta_0$, so that $\pi\zeta = 0$. 
\end{proof}

\begin{lem} \label{lem:replacement_of_reference_points}
In the setup of Definition \ref{dfn:k_standard}, suppose that we have points $\overline{x}^{f_i}$ on $f_i$ and $\overline{x}^{\overline{f}_i}$ on $\overline{f}_i$ such that $\varphi_{i-1} < \overline{x}^{f_i} < x^{f_i}$ and $\varphi_{i-1} < \overline{x}^{\overline{f}_i} < x^{\overline{f}_i}$ for $i = 1, \ldots, k$ and the set $\{ \overline{x}^{f_1}, \overline{x}^{\overline{f}_1}; \ldots; \overline{x}^{f_k}, \overline{x}^{\overline{f}_k} \}$ is admissible for $\langle \varphi_0, \cdots, \varphi_k \rangle$. Then there are $H'\zeta \in C^{\ell_1, \ldots, \ell_q}_{p+q+1}(\varphi)$ and $H'd^1\zeta \in C^{\ell_1, \ldots, \ell_q}_{p+q}(\varphi)$ such that $d^1H'\zeta + H'd^1\zeta = \overline{\zeta} - \zeta$, where
\begin{multline*}
\overline{\zeta} =
\sum_{f_1, \ldots, f_k}
\sum_{\alpha}
N^{f_1, \ldots, f_k}_\alpha
\langle 
\varphi_0, \overline{x}^{f_1} - \overline{x}^{\overline{f}_1}, \varphi_1, 
\cdots \\
\cdot\cdot, 
\varphi_{k-1}, \overline{x}^{f_k} - \overline{x}^{\overline{f}_k}, \varphi_k,
x^{\alpha, k+1}_1, \ldots, x^{\alpha, k+1}_{m^\alpha_{k+1}}, \varphi_{k+1},
\cdots,
\varphi_{q-1}, x^{\alpha, q}_1, \ldots, x^{\alpha, q}_{m^\alpha_q}, \varphi_q
\rangle.
\end{multline*}
\end{lem}

\begin{proof}
Notice that $\overline{\zeta} \in C^{\ell_1, \ldots, \ell_q}_{p+q}(\varphi)$ makes sense because of the assumptions $\varphi_{k-1} < \overline{x}^{f_k} < x^{f_k}$ and $\varphi_{k-1} < \overline{x}^{\overline{f}_k} < x^{\overline{f}_k}$. Notice also that we have
\begin{align*}
&\overline{x}^{f_i} \not< \varphi_i \not< x^{f_{i+1}}, &
&\overline{x}^{f_i} \not< \varphi_i \not< x^{\overline{f}_{i+1}}, \\
&\overline{x}^{\overline{f}_i} \not< \varphi_i \not< x^{f_{i+1}}, &
&\overline{x}^{\overline{f}_i} \not< \varphi_i \not< x^{\overline{f}_{i+1}},
\end{align*}
for $i = 1, \ldots, k-1$. The elements $\zeta$ and $d^1\zeta$ are linear combinations of the elements of the following form
\begin{align*}
\xi =
\langle 
\varphi_0, x^{f_1} - x^{\overline{f}_1}, \varphi_1, 
\cdots, \varphi_{k-1}, x^{f_k} - x^{\overline{f}_k}, \varphi_k,
x^{k+1}_1, \cdots
\rangle.
\end{align*}
For such an element $\xi$ and $j = 1, \ldots, k$, we define $H'_j\xi$ by
\begin{align*}
H'_j\xi
&=
\langle \varphi_0, 
\overline{x}^{f_1} - \overline{x}^{\overline{f}_1},
\varphi_1, \cdots, 
\varphi_{i-2},
\overline{x}^{f_{i-1}} - \overline{x}^{\overline{f}_{i-1}}, \\
&\quad
\varphi_{i-1}, 
\overline{x}^{f_i}, x^{f_i}, 
\varphi_i, 
x^{f_{i+1}} - x^{\overline{f}_{i+1}},
\varphi_{i+1},
\cdots, \varphi_{k-1}, x^{f_k} - x^{\overline{f}_k}, \varphi_k,
x^{k+1}_1, \cdots
\rangle \\
&-
\langle \varphi_0, 
\overline{x}^{f_1} - \overline{x}^{\overline{f}_1},
\varphi_1, \cdots, 
\varphi_{i-2},
\overline{x}^{f_{i-1}} - \overline{x}^{\overline{f}_{i-1}}, \\
&\quad
\varphi_{i-1}, 
\overline{x}^{\overline{f}_i}, x^{\overline{f}_i}, 
\varphi_i, 
x^{f_{i+1}} - x^{\overline{f}_{i+1}},
\varphi_{i+1},
\cdots, \varphi_{k-1}, x^{f_k} - x^{\overline{f}_k}, \varphi_k,
x^{k+1}_1, \cdots
\rangle.
\end{align*}
We now get $H'_j\zeta \in C^{\ell_1, \ldots, \ell_q}_{p+q+1}(\varphi)$ and $H'_jd^1\zeta \in C^{\ell_1, \ldots, \ell_q}_{p+q}(\varphi)$ by extending the definition above linearly. For the elements $H'\zeta = \sum_j H'_j\zeta$ and $H'd^1\zeta = \sum_j H'_jd^1\zeta$, we can verify $d^1H'\zeta + H'd^1\zeta = \overline{\zeta} - \zeta$. (cf.\ the proof of Lemma \ref{lem:homologous_cyclces})
\end{proof}

\begin{lem} \label{lem:induction_main_step}
Let $(X, d)$ be a geodesic space satisfying Assumption \ref{assumption:non_branching}, $q \ge 2$ an integer, $\ell_1,\ldots, \ell_q$ positive real numbers, and $\varphi = \langle \varphi_0, \cdots, \varphi_q \rangle \in \widehat{P}_q(X)$ a $q$-chain. Suppose that an element $\zeta \in C^{\ell_1, \ldots, \ell_q}_{p+q}(\varphi)$ is $k$-standard, where $1 \le k \le q-1$.
\begin{itemize}
\item[(a)]
If $\ell_{k+1} > d(\varphi_k, \varphi_{k+1})$, then there are $H\zeta \in C^{\ell_1, \ldots, \ell_q}_{p+q+1}(\varphi)$ and $Hd^1\zeta \in C^{\ell_1, \ldots, \ell_q}_{p+q}(\varphi)$ such that $d^1H\zeta + Hd^1\zeta = - \zeta$.

\item[(b)]
If $\ell_{k+1} = d(\varphi_k, \varphi_{k+1})$, then there are $H\zeta \in C^{\ell_1, \ldots, \ell_q}_{p+q+1}(\varphi)$ and $Hd^1\zeta \in C^{\ell_1, \ldots, \ell_q}_{p+q}(\varphi)$ such that $d^1H\zeta + Hd^1\zeta = \pi\zeta - \zeta$, where $\pi\zeta \in C^{\ell_1, \ldots, \ell_q}_{p+q}(\varphi)$ is $(k+1)$-standard. In particular, $\pi\zeta = 0$ if $p = k$.

\end{itemize}
\end{lem}

\begin{proof}
The proof will be a generalization of Lemma \ref{lem:induction_initial_step}, in which we need Lemma \ref{lem:replacement_of_reference_points} to resolve an issue present in the case of $k \ge 2$. We first prove (b). In this case, focusing on the number of points which follow $\varphi_k$, we can express $\zeta$ as $\zeta = \zeta_0 + \zeta_1 + \cdots$, where $\zeta_0$ and $\zeta_j$ for $j \ge 1$ have the expressions
\begin{align*}
\zeta_0 
&= \sum_{f_1, \ldots, f_k}
\sum_{\alpha \in \A_0}
N^{f_1, \ldots, f_k}_\alpha
\langle \cdots, 
\varphi_{k-1}, x^{f_k} - x^{\overline{f}_k}, \varphi_k,
\varphi_{k+1},
x^{\alpha, k+2}_1, 
\cdots
\rangle, \\
\zeta_j 
&= \sum_{f_1, \ldots, f_k} 
\sum_{f_{k+1}}
\sum_{\beta \in \B_j}
N^{f_1, \ldots, f_{k+1}}_{\beta_j}
\langle \cdots, 
\varphi_{k-1}, x^{f_k} - x^{\overline{f}_k}, \varphi_k,
x^{f_{k+1}, \beta, k+1}_1, \cdots, \\
&\quad\quad\quad\quad\quad
\quad\quad\quad\quad\quad
\quad\quad\quad\quad\quad
\quad\quad\quad\quad\quad
\cdots, x^{f_{k+1}, \beta, k+1}_j,
\varphi_{k+1},
\cdots
\rangle.
\end{align*}
In the above, each $f_i$ runs over a finite set in $\F^i_j \subset \Geod(\varphi_{i-1}, \varphi_i)$, $\A_0$ and $\B_j$ are finite sets, and the points $x^{f_{k+1}, \beta, k+1}_i$ lie on $f_{k+1}$. As in the proof of Lemma \ref{lem:induction_initial_step}, we can choose $\overline{x}^{f_{k+1}} \in f_{k+1}((0, \ell_{k+1}))$ for $f_{k+1} \in \F^{k+1} = \{ \overline{f}_{k+1} \} \cup \bigcup_j \F^{k+1}_j$ such that $\varphi_k < \overline{x}^{f_{k+1}} < x^{f_{k+1}, \beta, k+1}_1$ for all $\beta \in \bigcup_{j \ge 1}\B_j$ and $\overline{x}^{f_{k+1}} \not< \varphi_{k+1} \not< x^{\alpha, k+2}_1$ for all $\alpha \in \A_0$. Notice here that the relations $x^{f_k} \not< \varphi_k \not< \overline{x}^{f_{k+1}}$ may not be satisfied by $\overline{x}^{f_{k+1}}$, which is the issue present in the case that $j \ge 2$. However, we can show $\varphi_{k-1} \not< \varphi_k \not< \overline{x}^{f_{k+1}}$ under Assumption \ref{assumption:non_branching}. Hence we can find $\overline{x}^{f_k}$ on $f_k$ so as to be $\varphi_{k-1} < \overline{x}^{f_k} < x^{f_k}$ and $\overline{x}^{f_k} \not< \varphi_k \not< \overline{x}^{f_{k+1}}$ for all relevant $f_k$ and $f_{k+1}$. By Lemma \ref{lem:replacement_of_reference_points}, we have $H'\zeta \in C^{\ell_1, \ldots, \ell_q}_{p+q+1}(\varphi)$ and $H'd^1\zeta \in C^{\ell_1, \ldots, \ell_q}_{p+q}(\varphi)$ such that $H'd^1\zeta + H'd^1\zeta = \overline{\zeta} - \zeta$, where $\overline{\zeta}$ is given by replacing $x^{f_i}$ and $x^{\overline{f}_i}$ in $\zeta$ by $\overline{x}^{f_i}$ and $\overline{x}^{\overline{f}_i}$ respectively. For this $\overline{\zeta}$, we can directly generalize Lemma \ref{lem:induction_initial_step} by using $\overline{x}^{f_{k+1}}$, and get $H''\overline{\zeta} \in C^{\ell_1, \ldots, \ell_q}_{p+q+1}(\varphi)$ and $H''d^1\overline{\zeta} \in C^{\ell_1, \ldots, \ell_q}_{p+q}(\varphi)$ such that $d^1H''\overline{\zeta} + H''d^1\overline{\zeta} = \pi\overline{\zeta} - \overline{\zeta}$, in which $\pi\overline{\zeta}$ is $(k+1)$-standard and $\pi\overline{\zeta} = 0$ when $p = k$. In summary, (b) is satisfied by 
\begin{align*}
H\zeta &= H'\zeta + H''\overline{\zeta} 
\in C^{\ell_1, \ldots, \ell_q}_{p+q+1}(\varphi), &
Hd^1\zeta &= H'd^1\zeta + H''d^1\overline{\zeta}
\in C^{\ell_1, \ldots, \ell_q}_{p+q}(\varphi).
\end{align*}
Now, (a) is shown in the same way: Let $\zeta$ be as in Definition \ref{dfn:k_standard}. Under the assumption $d(\varphi_k, \varphi_{k+1}) > \ell_{k+1}$, we focus on the points which follow $\varphi_k$, and find some points $\overline{x}^\lambda$ so that we can apply the constructions in the proof of Lemma \ref{lem:replacable_by_chains_starting_with_non_4_cut} and Lemma \ref{lem:local_homotopy_for_E2_p1}. Though $x^{f_k} \not< \varphi_k \not< \overline{x}^\lambda$ may not be satisfied, Assumption \ref{assumption:non_branching} allows us to find points $\overline{x}^{f_i}$ such that we can apply the constructions by using $\overline{x}^\lambda$ to the element $\overline{\zeta}$ given by replacing $x^{f_i}$ in $\zeta$ by $\overline{x}^{f_i}$. The elements $\zeta$ and $\overline{\zeta}$ are related as in Lemma \ref{lem:replacement_of_reference_points}, and we finally get $H\zeta$ and $Hd^1\zeta$ such that $d^1H\zeta + Hd^1\zeta = - \zeta$. 
\end{proof}

\begin{prop} \label{prop:determine_all_direct_summand}
Let $(X, d)$ be a geodesic space satisfying Assumption \ref{assumption:non_branching}, $q \ge 2$ an integer, $\ell_1,\ldots, \ell_q$ positive real numbers, and $\varphi = \langle \varphi_0, \cdots, \varphi_q \rangle \in \widehat{P}_q(X)$.
\begin{itemize}
\item[(a)]
$H^{\ell_1, \ldots, \ell_q}_{p+q}(\varphi) = 0$ if $\ell_i > d(\varphi_{i-1}, \varphi_i)$ for some $i$ or $p \neq q$.

\item[(b)]
In the case that $\ell_i = d(\varphi_{i-1}, \varphi_i)$ for $i = 1, \ldots, q$ and $p = q$, we choose $\overline{f}_i \in \Geod(\varphi_{i-1}, \varphi_i)$ for each $i$. Then we have
$$
H^{\ell_1, \ldots, \ell_q}_{2q}(\varphi)
= \bigoplus_{(f_i) \neq (\overline{f}_i)} 
\Z \Gamma(f_1, \overline{f}_1; \cdots; f_q; \overline{f}_q),
$$
where $(f_i) = (f_1, \ldots, f_q)$ runs over $\Geod(\varphi_0, \varphi_1) \times \cdots \times \Geod(\varphi_{q-1}, \varphi_q)$ such that $f_i \neq \overline{f}_i$ for $i = 1, \ldots, q$. 
\end{itemize}
\end{prop}

\begin{proof}
We apply Lemma \ref{lem:induction_initial_step} and then Lemma \ref{lem:induction_main_step} inductively to a given cycle $\zeta \in C^{\ell_1, \ldots, \ell_q}_{p+q}(\varphi)$. If $\ell_i > d(\varphi_{i-1}, \varphi_i)$ for some $i$ or $p \neq q$, then $\zeta$ turns out to be null-homologous at some step, so that (a) follows. If $\ell_i = d(\varphi_{i-1}, \varphi_i)$ for all $i$ and $p = q$, then we finally see that $\zeta$ is homologous to a linear combination of the cycles representing $\Gamma(f_1, \overline{f}_1; \cdots; f_q, \overline{f}_q)$. Hence these homology classes form a generating set of $H^{\ell_1, \ldots, \ell_q}_{2q}(\varphi)$. 

To complete the proof of (b), we shall show that this generating set is a basis. To make the argument clear, let $\Z \tilde{\Gamma}$ be the free abelian group generated by the formal symbols $\tilde{\Gamma}(f_1; \cdots; f_q)$ associated to $f_i \in \Geod(\varphi_{i-1}, \varphi_i)$ such that $f_i \neq \overline{f}_i$. There is an obvious homomorphism
\begin{align*}
\varpi &: \Z\tilde{\Gamma} \to H^{\ell_1, \ldots, \ell_q}_{2q}(\varphi), &
\varpi(\tilde{\Gamma}(f_1; \cdots; f_q))
&= \Gamma(f_1, \overline{f}_1; \cdots; f_q; \overline{f}_q).
\end{align*}
This is surjective, since we already know the generating set of $H^{\ell_1, \ldots, \ell_q}_{2q}(\varphi)$. To see that $\varpi$ is injective, we construct a homomorphism $\varrho : H^{\ell_1, \ldots, \ell_q}_{2q}(\varphi) \to \Z\tilde{\Gamma}$ as follows. Let $Z^{\ell_1, \ldots, \ell_q}_{2q}(\varphi) \subset C^{\ell_1, \ldots, \ell_q}_{2q}(\varphi)$ be the group of cycles. For a cycle $\zeta \in Z^{\ell_1, \ldots, \ell_q}_{2q}(\varphi)$ given, we can express it as $\zeta = \zeta' + \zeta''$, where the chains in $\zeta'$ belong to $C^{\ell_1, \ldots, \ell_q}_{2, \ldots, 2}(\varphi)$, and those in $\zeta''$ do not. We can then express $\zeta'$ in the following way
$$
\zeta' = 
\sum_{f_1, \ldots, f_q} 
\sum_{i_{f_1}, \ldots, i_{f_q}}
N^{f_1, \ldots, f_q}_{i_{f_1}, \ldots, i_{f_q}}
\langle 
\varphi_0, x^{f_1}_{i_{f_1}}, \varphi_1, \cdots,
\varphi_{q-1}, x^{f_q}_{i_{f_q}}, \varphi_q
\rangle,
$$
where $f_j$ runs over (a finite set in) $\Geod(\varphi_{j-1}, \varphi_j)$, the index $i_{f_j}$ over a finite set for each $f_j$, and the point $x^{f_j}_{i_{f_j}}$ lies on $f_j$. We define $\tilde{\varrho} : Z^{\ell_1, \ldots, \ell_q}_{2q}(\varphi) \to \Z\tilde{\Gamma}$ by 
$$
\tilde{\varrho}(\zeta) = \tilde{\varrho}(\zeta')
= 
\sum_{f_1 \neq \overline{f}_1, \ldots, f_q \neq \overline{f}_q}
\sum_{i_{f_1}, \ldots, i_{f_q}}
N^{f_1, \ldots, f_q}_{i_{f_1}, \ldots, i_{f_q}}
\tilde{\Gamma}(f_1; \cdots; f_q).
$$
Let us consider here a chain of the following form
$$
\beta
=
\langle \cdots, \varphi_{j-2}, x^{f_{j-1}}, \varphi_{j-1}, 
\overline{x}^{f_j}, x^{f_j}, \varphi_j,
x^{f_{j+1}}, \varphi_{j+1}, \cdots 
\rangle
\in C^{\ell_1, \ldots, \ell_q}_{2q+1}(\varphi),
$$
where $x^{f_j}$ and $\overline{x}^{f_j}$ lie on $f_j \in \Geod(\varphi_{j-1}, \varphi_j)$. The part of $d^1\zeta$ that belongs to $C^{\ell_1, \ldots, \ell_q}_{2, \ldots, 2}(\varphi)$ is
\begin{multline*}
-
\langle \cdots, \varphi_{j-2}, x^{f_{j-1}}, \varphi_{j-1}, 
x^{f_j}, \varphi_j,
x^{f_{j+1}}, \varphi_{j+1}, \cdots 
\rangle \\
+
\langle \cdots, \varphi_{j-2}, x^{f_{j-1}}, \varphi_{j-1}, 
\overline{x}^{f_j}, \varphi_j,
x^{f_{j+1}}, \varphi_{j+1}, \cdots 
\rangle.
\end{multline*}
Because $x^{f_j}$ and $\overline{x}^{f_j}$ line on the same geodesic $f_j$, it follows that $\tilde{\varrho}(d^1\beta) = 0$. The boundary of a chain in $C^{\ell_1, \ldots, \ell_q}_{2q+1}(\varphi)$ which is not the form of $\beta$ above contributes trivially to $\tilde{\varrho}$. Hence $\tilde{\varrho}$ is trivial on the subgroup $d^1(C^{\ell_1, \ldots, \ell_q}_{2q+1}(\varphi))$ of boundaries. As a result, $\tilde{\varrho}$ descends to induce a homomorphism $\varrho$ in question. It is easy to see that $\varrho \circ \varpi$ is the identity homomorphism on $\Z\tilde{\Gamma}$. Therefore $\varpi$ is injective.
\end{proof}

Now, Theorem \ref{thm:main} directly follows from the proposition above.

\begin{thm} \label{thm:complete_description_under_non_branching}
Let $(X, d)$ be a geodesic space satisfying Assumption \ref{assumption:non_branching}, and $H^\ell_n(X)$ the $n$th magnitude homology of $X$ with length $\ell > 0$. 
\begin{enumerate}
\item[(a)]
If $n$ is odd, then $H^\ell_n(X) = 0$ for any $\ell$.

\item[(b)]
If $n = 2q$ is even, then
$$
H^\ell_n(X)
\cong
\bigoplus_{\ell_1, \ldots, \ell_q}
\bigoplus_{\varphi_0, \ldots, \varphi_q}
\bigoplus_{f_1, \ldots, f_q}
\Z (f_1, \ldots, f_q).
$$
In the above, the direct sum is taken over:
\begin{itemize}
\item
positive real numbers $\ell_1, \ldots, \ell_q$ such that $\ell = \ell_1 + \cdots + \ell_q$, 

\item
points $\varphi_0, \ldots, \varphi_q \in X$ such that $d(\varphi_{i-1}, \varphi_i) = \ell_i$ for $ = 1, \ldots, q$,

\item
geodesics $f_i \in \Geod(\varphi_{i-1}, \varphi_i)$ such that $f_i \neq \overline{f}_i$ for $i = 1, \ldots, q$, 
\end{itemize}
where $\overline{f}_i \in \Geod(\varphi_{i-1}, \varphi_i)$ are geodesics chosen arbitrarily as references. We mean by $\Z(f_1, \ldots, f_q)$ the free abelian group of rank $1$ generated by the formal symbol $(f_1, \ldots, f_q)$.
\end{enumerate}
\end{thm}

\begin{proof}
Recall the direct sum decomposition of the spectral sequence
$$
E^2_{p, q}
= 
\bigoplus_{\varphi \in \widehat{P}_q(X)}
\bigoplus_{\stackrel{\ell_1, \ldots, \ell_q > 0}
{\ell_1 + \cdots + \ell_q = \ell}}
H^{\ell_1, \ldots, \ell_q}_{p+q}(\varphi).
$$
By Proposition \ref{prop:E2_p1} and Proposition \ref{prop:determine_all_direct_summand}, we have $E^2_{p, q} = 0$ if $p \neq q$. This proves (a), and also the $E^2$-degeneracy of the spectral sequence. Using Proposition \ref{prop:determine_all_direct_summand} again, we get
$$
H^\ell_{2q}(X)
= E^2_{q, q} 
= \bigoplus_{\varphi \in P_q(X)}
\bigoplus_{\stackrel{\ell_1, \ldots, \ell_q > 0}
{\ell_1 + \cdots + \ell_q = \ell}}
H^{\ell_1, \ldots, \ell_q}_{2q}(\varphi).
$$
Describing $H^{\ell_1, \ldots, \ell_q}_{2q}(\varphi)$ in terms of the basis, we complete the proof of (b).
\end{proof}

\begin{cor} \label{cor:complete_description_under_non_branching:torsion}
If $(X, d)$ is a geodesic space satisfying Assumption \ref{assumption:non_branching}, then its magnitude homology $H^\ell_n(X)$ is torsion free for any $n$ and $\ell$.
\end{cor}

\begin{cor} \label{cor:vanishing_on_uniquely_geodesic_space}
Let $(X, d)$ be a uniquely geodesic space, and $H^\ell_n(X)$ its magnitude homology. If $n \neq 0$ or $\ell \neq 0$, then $H^\ell_n(X) = 0$.
\end{cor}

Since the geodesic spaces satisfying Assumption \ref{assumption:non_branching} cover connected and complete Riemannian manifolds, we can compute their magnitude homology groups by using Theorem \ref{thm:complete_description_under_non_branching} and knowledge of geodesics. As an example, let us consider the circle $S^1$ of radius $r$ equipped with geodesic distance. Any distinct points $a, b \in S^1$ such that $d(a, b) < \pi r$ have a unique geodesic joining $a$ to $b$. For any point $a \in S^1$, there is a unique point $\check{a} \in S^1$ such that $d(a, \check{a}) = \pi r$, and there are two distinct geodesics joining $a$ to $\check{a}$. With this knowledge and Theorem \ref{thm:complete_description_under_non_branching}, we easily get
$$
H^\ell_n(S^1)
\cong
\left\{
\begin{array}{ll}
\Z[S^1], & (\mbox{$n \ge 0$ even and $\ell = \pi r n/2$}) \\
0. & (\mbox{otherwise})
\end{array}
\right.
$$
More generally, let $S^d \subset \R^{d+1}$ be the sphere of dimension $d \ge 1$ with radius $r$, which we regard as a geodesic space by the standard Riemannian metric induced from $\R^{d+1}$. Any distinct points $a, b \in S^d$ such that $d(a, b) < \pi r$ admit a unique geodesic joining $a$ to $b$. For any $a \in S^d$, there is a unique point $\check{a} \in S^d$ such that $d(a, \check{a}) = \pi r$, and geodesics joining $a$ to $\check{a}$ are in one to one correspondence with points on $S^{d-1}$. Consequently, one has
$$
H^\ell_n(S^d)
\cong
\left\{
\begin{array}{ll}
\Z[S^d \times \underbrace{\dot{S}^{d-1} \times 
\cdots \times \dot{S}^{d-1}}_{n/2}],&
 (\mbox{$n \ge 0$ even and $\ell = \pi r n/2$}) \\
0, & (\mbox{otherwise})
\end{array}
\right.
$$
where $\dot{S}^{d-1}$ is $S^{d-1}$ with one point removed.


\end{document}